\documentclass[11pt]{article}
\usepackage[utf8]{inputenc}
\usepackage{amsfonts,amsmath,amsthm,amssymb}
\usepackage{dsfont}
\usepackage{fullpage}
\usepackage{xcolor}
\def\clb{\color{blue}}

\usepackage{ulem}
\usepackage{hyperref}
\usepackage{framed}
\usepackage{graphicx}
\usepackage{natbib}
\usepackage{mathtools}
\usepackage{ulem}
\mathtoolsset{showonlyrefs} 
\allowdisplaybreaks

\newcommand \comment[1]{{\fbox{{\color{red}Commented out material}}}}
\def\P{\mathbb P}
\def\E{\mathbb E}

\newtheorem{theorem}{Theorem}[section]
\newtheorem{prop}[theorem]{Proposition}
\newtheorem{cor}[theorem]{Corollary}
\newtheorem{lemma}[theorem]{Lemma}
\theoremstyle{definition}
\newtheorem{definition}[theorem]{Definition}
\newtheorem{remark}[theorem]{Remark}

\numberwithin{equation}{section}

\def\ToM0{\stackrel{{\rm M_0}}{\Longrightarrow}}
\def\toM0{\Rightarrow^{\rm M_0}}

\def\R{\mathbb R}

\title{Tail-dependence, exceedance sets, and metric embeddings}
\author{Anja Jan{\ss}en\footnote{Otto-von-Guericke University Magdeburg, Germany (anja.janssen@ovgu.de)},\quad Sebastian Neblung\footnote{University of Hamburg, Germany (sebastian.neblung@uni-hamburg.de)},\quad Stilian Stoev\footnote{University of Michigan, Ann Arbor, US (sstoev@umich.edu); This author was partially supported by the NSF Grant DMS-1916226.} }

\begin{document}
\parindent0cm

\maketitle

\begin{abstract}

  There are many ways of measuring and modeling tail-dependence in random vectors: from the general framework of multivariate regular variation and the flexible class of max-stable vectors down to simple and concise summary measures like the matrix of bivariate tail-dependence coefficients. This paper starts by providing a review of existing results from a unifying perspective, which highlights connections between extreme value theory and the theory of cuts and metrics. Our approach leads to some new findings in both areas with some applications to current topics in risk management. 
  
  We begin by using the framework of multivariate regular variation to show that extremal coefficients, or equivalently, the higher-order tail-dependence coefficients of a random vector can simply be understood in terms of random exceedance sets, which allows us to extend the notion of Bernoulli compatibility. In the special but important case of bi-variate tail-dependence, we establish a correspondence between tail-dependence matrices and
   $L^1$- and $\ell_1$-embeddable finite metric spaces via the spectral distance, which is a metric on the space 
   of jointly $1$-Fr\'{e}chet random variables. Namely, the coefficients of the cut-decomposition of 
   the spectral distance and of the Tawn-Molchanov max-stable model realizing the 
   corresponding bi-variate extremal dependence coincide. We show that line metrics are rigid and if the spectral distance corresponds to a line metric, the higher order tail-dependence is determined by the bi-variate tail-dependence matrix. 
   
   Finally, the correspondence between $\ell_1$-embeddable metric spaces and tail-dependence matrices allows us
   to revisit the realizability problem, i.e.\ checking whether a given matrix is a valid 
   tail-dependence matrix. We confirm a conjecture of \cite{Shyamalkumar2020} that this problem is NP-complete.

    
\end{abstract}

\section{Introduction}\label{sec:Introduction}
Extreme events such as large portfolio losses in insurance and finance, spatial and 
environmental extremes such as heat-waves, floods, electric grid outages, and many other
complex system failures are associated with tail-events.  That is, the simultaneous occurrence 
of extreme values in the components of a possibly very high-dimensional 
vector $X = (X_i)_{1\leq i\leq p}$ of covariates.  
Such simultaneous extremes occur due to \textit{dependence} among the extremes of the $X_i$'s.
This has motivated a large body of literature on modeling and quantifying tail-dependence, see, e.g.\ \citep{coles:2001,finkenstaedt:2003,rachev:2003,beirlant:2004,castillo:2012, resnick:2007,deHaanFerreira:2007}.  One basic and popular
measure is the bivariate (upper) tail-dependence coefficient
\begin{equation}\label{e:lambda-bivariate-classic}
\lambda_X(i,j):= \lim_{ u \uparrow 1} \P [ X_i > F_i^{-1}(u) | X_j > F_j^{-1}(u)],
\end{equation}
where $F_i^{-1}(u):= \inf\{ x\, :\, \P[X_i \le x] \ge u\}$ is the generalized inverse 
of the cumulative distribution function $F_i$ of $X_i$.  Under weak conditions 
the above limit exists and is independent of the choice of the (continuous) marginal 
distributions of $(X_i,X_j)$. The matrix $\Lambda := (\lambda_X(i,j))_{p\times p}$ of 
bivariate tail-dependence coefficients is necessarily positive (semi)definite and in fact, 
since $\lambda_X(i,i) = 1$, it is a correlation matrix of a random vector, see \cite{schlather:tawn:2003}. We call $\Lambda$ as defined in \eqref{e:lambda-bivariate-classic} a \emph{tail-dependence matrix} or \emph{TD matrix} for short.\\

The general theme of our paper is that we review and contribute to the unified treatment of 
tail-dependence using the powerful framework of multivariate regular variation.  This leads 
to deep connections
to existing results in the theory of cut (semi)metrics and $\ell_1$-embeddable metrics \citep{deza:1997}, as
well as to extensions to the {\em Bernoulli compatibility} characterization of tail-dependence matrices
established in \cite{embrechts:hofert:wang:2016} and \cite{krauseetal2018}. What follows is an overview 
of our key ideas and contributions.\\

Since the marginal distributions of $X$ are not important in quantifying tail-dependence, one may transform its 
marginals to be heavy-tailed.  In fact, we make the additional and often very mild assumption that the vector $X$
is regularly varying, i.e., that there exists a Radon measure $\mu$ on $\R^p\setminus\{0\}$ and a suitable positive sequence $a_n\uparrow \infty$ such that
$$
n\P[X \in a_n A] \to \mu(A),\ \ \mbox{as } n\to\infty,
$$
for all Borel sets $A\subset \R^p$ that are bounded away from $0$ and such that 
$\mu(\partial A) =0$ (see Definition \ref{d:RV}, below).  This allows us to conclude that
$n\P[ h(X) > a_n] \to \mu\{ h>1\}$ for a large class of continuous and $1$-homogeneous 
functions $h:\R^p \to [0,\infty)$ (Proposition \ref{p:h(Y)}).  Therefore, if $h$ is a certain {\em risk functional}, we 
readily obtain an asymptotic approximation of the probability of an extreme loss 
$\P[h(X) > a_n] \approx n^{-1} \mu\{h>1\}$.  By varying the risk functional $h$, one obtains 
different measures of tail-dependence, which may be of particular interest to 
practitioners.  For example, if $L = \{i_1,\cdots,i_k\}\subset[p]:=\{1,\cdots,p\}$ and taking $h_L(X) = (\min_{i \in L} X_i)_+:=\max\{0,\min_{i \in L} X_i\}$, the risk functional
quantifies the joint exceedance probability 
$$
\P[h_L(X) > a_n] = \P[\min_{i\in L} X_i>a_n]
$$
that all components of $X$ with index in the set $L$ are simultaneously 
extreme -- an event with potentially devastating consequences.  In practice, due to the limited horizon of
historical data such extreme events especially for large sets $L$ are rarely (if ever) 
observed.  Thus, quantifying their probabilities is very challenging. Yet, as Emil Gumbel 
had eloquently put it  {\it ``It is not possible that the improbable will never occur.''} This underscores the
importance of the theoretical understanding, modeling, and inference of such functionals.  Namely, one naturally arrives at the higher order {\em tail-dependence coefficients}
\begin{equation}\label{e:lambda-intro}
\lambda_X(L):= \lim_{n\to\infty} n \P[ \min_{i\in L} X_i>a_n].    
\end{equation}

It can be seen that if the marginals of the $X_i$'s are identical and $a_n$ is such that
$n^{-1} \sim \P[X_i>a_n]$ (i.e. $\lim_{n \to \infty} n\P[X_i>a_n]=1$), then $\lambda_X(\{i,j\}) = \lim_{n\to\infty} \P[ X_i>a_n\mid X_j>a_n]$ recovers the classic bivariate tail-dependence coefficients $\lambda_X(i,j)$ in \eqref{e:lambda-bivariate-classic}.  Using the functionals $h(X):= \max_{j\in K} X_j$ for 
some $K\subset[p]$, one
arrives at the popular {\em extremal coefficients} 
arising in the study of max-stable processes:
\begin{equation}\label{e:theta-intro}
\theta_X(K):= \lim_{n\to\infty} n\P[\max_{j\in K} X_j > a_n].    
\end{equation}

Starting from the seminal works of \cite{Schlather:2002,schlather:tawn:2003}, 
the structure of the extremal coefficients $\{\theta_X(K),\ K\subset [p]\}$ has been studied
extensively, see \cite{Strokorb:2015-extremal,Strokorb:2015-tcf,molchanov:strokorb:2016,fiebig:2017}, which address fundamental theoretical problems and develop stochastic process extensions. Our goal here is more modest.  We want to study \textit{both} the 
tail-dependence \textit{and} extremal coefficients as risk functionals from the unifying 
perspective of regular variation.  Interestingly, they can be succinctly understood in terms of
{\em exceedance sets.}  Namely, defining the random set
$$
\Theta_n:=\{ i\in [p]\, :\, X_i> a_n\}
$$
we show (Proposition \ref{p:Theta-set} below)
$$
\Theta_n | \{\Theta_n\not=\emptyset\} \stackrel{d}{\longrightarrow} \Theta,\ \ 
\mbox{ as }n\to\infty,
$$
where the limit $\Theta$ is a non-empty random subset of $[p]$ such that
\begin{equation} \label{e:lambda-theta-via-Theta-intro}
\lambda_X(L) = a\cdot \P[ L \subset \Theta]\ \ \mbox{ and }\ \ \theta_X(K) = a\cdot \P[K\cap \Theta \not=\emptyset],
\end{equation}
where $a = \theta_X([p])$.  Thus, $\lambda_X$ and $\theta_X$ (up to rescaling by $a$) are precisely the inclusion and hitting functionals characterizing the distribution of 
$\Theta$ \citep{molchanov:2017}.  Interestingly, the probability mass function of the random set $\Theta$ recovers (up to rescaling) the coefficients in a (generalized) 
Tawn-Molchanov max-stable model associated with $X$ (see \eqref{e:Theta-PMF}).

The above probabilistic representation in \eqref{e:lambda-theta-via-Theta-intro} of the 
tail-dependence functionals leads to transparent proofs of seminal results from
\cite{embrechts:hofert:wang:2016} and \cite{krauseetal2018} on the characterization 
of TD matrices in terms of so-called Bernoulli-compatible matrices. In fact, we readily 
obtain a more general result on the characterization of higher-order tail-dependence
coefficients via Bernoulli-compatible tensors (Proposition \ref{p:Bernoulli-arrays}).\\

Associated to the bivariate tail-dependence coefficients $\lambda_X(\{i,j\})$ we introduce and discuss the so called {\em spectral distance} $d_X$ given by 
$$d_X(i,j):=\lambda_X(\{i\})+\lambda_X(\{j\})-2\lambda_X(\{i,j\}).$$ 
This spectral distance defines a metric on the space of $1$-Fr\'echet random variables (i.e.\ random variables with distribution function $F(x)=\exp\{-c/x\}, x \geq 0,$ for some non-negative scale coefficient $c$, where we speak of a standard $1$-Fr\'echet distribution if $c=1$) living on a joint probability space, which metricizes convergence in probability and was considered in \cite{davis:1993,stoev:taqqu:2005,fiebig:2017}. In Section \ref{sec:l1-emb} we will establish the $L^1$-embeddability of this metric, which allows us to apply the rich theory about metric embeddings in the context of analyzing the tail-dependence coefficients.

In Section \ref{sec:line-tm}, utilizing the exceedence set representation of the bivariate tail-dependence 
coefficients and the $L^1$-embeddability of the spectral distance, we recover the equivalence of the $L^1$ and 
$\ell_1$-embeddability as well as a probabilistic proof of the so-called cut-decomposition of $\ell_1$-embeddable 
finite metric spaces. In this case, this decomposition turns out to be closely related to the Tawn-Molchanov model of an associated max-stable vector $X$ (Proposition \ref{p:d-via-cuts}).
When a given $\ell_1$-embeddable metric has a unique cut-decomposition, it is called {\em rigid} \citep{deza:1997}. Rigidity of the spectral distance basically means that the bivariate tail-dependence coefficients $\Lambda$ determine all higher order tail-dependence coefficients. 
In Theorem \ref{thm:line-tm}, we show that {\em line metrics} are rigid, which to the best of our knowledge
is a new finding. In particular, we obtain that the bivariate tail-dependence coefficient
matrices corresponding to line metrics determine the complete set of tail-dependence or, equivalently,
extremal coefficients of $X$. Interestingly, the random set $\Theta$ corresponding to such
line-metric tail-dependence is (after a suitable reordering of marginals) a random segment, more precisely a random set of the form $\{i, i+1, \ldots, j-1, j\}$ for $1 \leq i \leq j \leq p$ with $i=1$ or $j=p$.  In general, the characterization of rigidity is computationally hard as it is equivalent to the  characterization of the simplex faces of the cone of cut metrics \citep{deza:1997}.

The bivariate TD matrix $\Lambda$ is a correlation matrix of a random vector. It is well-known, however, that not every correlation matrix with non-negative entries is a matrix of tail-dependence coefficients.
The recent works of \cite{fiebig:2017}, \cite{embrechts:hofert:wang:2016}, \cite{krauseetal2018}, and \cite{Shyamalkumar2020} among others have studied extensively various aspects of the class of TD matrices. One surprisingly difficult problem, referred to as the {\em realizability problem}, is checking whether a given matrix $\Lambda$ is a valid TD matrix.  
The extensive study of \cite{Shyamalkumar2020} proposed several practical and efficient algorithms for realizability. Moreover, \cite{Shyamalkumar2020} conjectured that the realizability problem is NP-complete. In Section \ref{sec:NP}, we confirm their conjecture.  We do so by exploiting the established connection to $\ell_1$-embeddability, which allows us to utilize the rich theory on cuts and metrics outlined in the monograph of \cite{deza:1997}. 
It is known that checking whether any given $p$-point metric space is $\ell_1$-embeddable is a computationally hard problem in the NP-complete class. 
\\

{\em The paper is structured as follows:}
In Section~\ref{sec:notation} we give an overview over several ways of modeling and measuring tail-dependence of a random vector, presented in a hierarchic fashion: First of all, multivariate regular variation allows for the most complete asymptotic description of the tail-behavior of (heavy-tailed) random vectors in terms of the tail measure, with a direct correspondence to the class of max-stable models as the natural representatives for each given tail measure. A more condensed description of tail-dependence is given by the values of special extremal dependence functionals like the extremal coefficients and tail-dependence coefficients. Finally, a rather coarse but popular description of the tail-dependence is given in form of those functions evaluated only at bivariate marginals, where the bivariate tail-dependence coefficients form the most prominent example. 

In Section~\ref{sec:TM} we first discuss exceedance sets, as introduced above, and Bernoulli compatibility. Based on this interpretation we give a short introduction into generalized Tawn-Molchanov models. 

In Section~\ref{sec:l1-emb} we explore the relationship between bivariate tail-dependence coefficients and the spectral distance on the space of $1$-Fr\'{e}chet random variables. After a brief introduction into the concepts of metric embeddings of finite metric spaces we will show that the spectral distance is both $L^1$- and $\ell_1$-embeddable, some consequences of which will be explored in Section \ref{sec:line-tm} and Section \ref{sec:NP}. In Section \ref{sec:line-tm} we introduce the concept of rigid metrics and prove that the building blocks of $\ell_1$-embeddability, i.e.\ the line metrics, correspond to Tawn-Molchanov models with a special structure which is completely determine by this line metric. 

Finally, in Section~\ref{sec:NP} we use known results about the computational complexity of embedding problems to show that the realization problem of a tail-dependence matrix is NP-complete.
Some proofs are deferred to the Appendix \ref{sec:proofs}.

\section{Regular variation, max-stability, and extremal dependence}\label{sec:notation}

In this section, we provide a concise overview of fundamental notions on multivariate regular
variation and max-stable distributions, which underpin the study of tail-dependence.  

\subsection{Multivariate regular variation}\label{subsec:MRV} 
The concept of multivariate regular variation is key to the unified treatment of the various tail-dependence notions we will consider.  Much of this material is classic but we provide here a self-contained review tailored to our purposes. Many more details and insights can be found in \cite{resnick:1987,resnick:2007,hult:lindskog:2006,basrak:planinic:2019,kulik:soulier:2020} among 
other sources. 

\medskip We start with a few notations. 
A set $A \subset\R^p$ is said to be \emph{bounded away from $0$} if 
$0\not \in A^{\rm cl}$, i.e., $A\cap B(0,\varepsilon) =\emptyset$, for some $\varepsilon>0$. Here $A^{\rm cl}$ is the closure of $A$ and $B(x,r):=\{ y\in \R^p\, :\, \|x-y\|<r\}$ is the
ball of radius $r$ centered at $x$ in a given fixed norm $\|\cdot\|$. Furthermore, denote the Borel $\sigma$-Algebra on $\mathbb{R}^p$ by ${\cal B}(\R^p)$.

Consider the class $M_0(\R^p)$ of all Borel measures $\mu$ on ${\cal B}(\R^p)$ that
are finite on sets bounded away from $0$, i.e., such that $\mu(B(0,\varepsilon)^c)<\infty$, for all $\varepsilon>0$. Such measures will be referred to as \emph{boundedly finite}. 
For $\mu_n,\mu\in M_0(\R^p),$ we write 
$$
\mu_n \ToM0 \mu,\ \ \mbox{ as }n\to\infty,
$$
if
$
\int_{\R^p} f(x)\mu_n(dx) \to \int_{\R^p} f(x) \mu(dx), \mbox{ as }n\to\infty,
$
for all bounded and continuous $f$ vanishing in a neighborhood of $0$. The latter
is equivalent to having
\begin{equation}\label{e:d:RV}
\mu_n(A) \to \mu(A),\ \ \mbox{ as }n\to\infty,
\end{equation}
for all $\mu$-continuity Borel sets $A$ that are bounded away from $0$ \citep[Theorems 2.1 and 2.4]{hult:lindskog:2006}.

\begin{definition}\label{d:RV} A random vector $X$ in $\R^p$ is said to be regularly varying
if there is a positive sequence $a_n\uparrow \infty$ and a non-zero $\mu\in M_0(\R^p)$ such that
$$
n \P[ X \in a_n \cdot ] \ToM0 \mu(\cdot),\ \ \mbox{ as }n\to\infty. 
$$ 
In this case, we write $X\in {\rm RV}(\{a_n\},\mu)$ and call $\mu$ the \emph{tail measure} of $X$.
\end{definition}

If $X\in {\rm RV}(\{a_n\},\mu)$, then it necessarily follows that there
is a {\em positive index} $\alpha>0$ such that
\begin{equation}\label{e:mu-scaling}
\mu( c A) = c^{-\alpha} \mu(A),\ \ \mbox{ for all $c>0$ and $A\in {\cal B}(\R^p)$,}
\end{equation}
and, moreover, $a_n \sim n^{1/\alpha} \ell(n)$, for some slowly varying 
function $\ell$, see, e.g., \cite{kulik:soulier:2020}, Section 2.1. We shall denote by
${\rm index}(X)$ the index of regular variation $\alpha$ and sometimes write $X\in {\rm RV}_\alpha(\{a_n\},\mu)$ to specify that ${\rm index}(X) = \alpha$.

The measure $\mu$ is unique up to a multiplicative constant and the scaling property \eqref{e:mu-scaling} implies that $\mu$ factors into a radial and an angular component.  Namely, fix any norm $\|\cdot\|$ in $\R^p\setminus\{0\}$
and define  the polar coordinates $r:=\|x\|$ and $u:= x/\|x\|,\ x\not = 0$. Then,
\begin{equation}\label{e:mu-spectral}
\mu(A) = \int_{S} \int_0^\infty  1_{A}(r u) \alpha r^{-\alpha-1}dr \sigma(du),
\end{equation}
where $S:=\{x\, :\, \|x\|=1\}$ is the unit sphere and $\sigma$ is a finite Borel
measure on $S$ referred to as the angular or spectral measure
associated with $\mu$, see, e.g., \cite{kulik:soulier:2020}, Section 2.2. Given the norm $\|\cdot\|$, the measure $\sigma$ is uniquely 
determined as
\begin{equation}\label{e:sigma-spectral}
    \sigma(B) = \mu(\{x\, :\, \|x\|>1,\ x/\|x\| \in B\}), \;\;\; B \in {\cal{B}}(S), 
\end{equation}
where $ {\cal B}(A)$ for $A \subset \mathbb{R}^d$ denotes the $d$-dimensional Borel sets which are also subsets of $A$.
The following is a useful characterization of regular variation sometimes taken
as an equivalent definition, see again, e.g., \cite{kulik:soulier:2020}, Section 2.2.
\begin{prop}\label{p:RV:via:exceed} We have $X\in {\rm RV}_\alpha(\{a_n\},\mu)$ if and only if for all $x>0$
$$
n \P[ \|X\|>a_n x]\to x^{-\alpha},\ \mbox{ as $n\to\infty$, and }\
\P\Big[ \frac{X}{\|X\|} \in \cdot \  |\  \|X\| >r \Big]\Longrightarrow \sigma(\cdot),\ \mbox{ as $r\to\infty$},
$$
where $\Rightarrow$ denotes the weak convergence of probability distributions.
\end{prop}
Proposition~\ref{p:RV:via:exceed} characterizes regularly varying random vectors in terms of exceedances over a threshold. An equivalent charaterization is also possible in terms of maxima, see, e.g., \cite{kulik:soulier:2020}, Section 2.1.
\begin{prop}\label{p:RV:via:maxima} For a random vector $Y \in [0,\infty)^d$ we have $ Y \in {\rm RV}_\alpha(\{a_n\},\mu)$ if and only if there exists a non-degenerate random vector $X$ such that for all $x \in [0,\infty)^d$
\begin{equation}\label{e:max-stable-limit}
\P\Big[ a_n^{-1} \bigvee_{t=1}^{n}Y^{(t)} \leq x \Big] \to \P[X \leq x]=\exp\{-\mu[0,x]^c\},\ \mbox{ as } n\to\infty,
\end{equation}
where $[0,x]^c:= \R_+^p \setminus [0,x]=\R_+^p \setminus ([0,x_1] \times \ldots \times [0,x_p])$ and $Y^{(t)},\ t=1,\dots,n$ are independent copies of $Y$ and the operation $\vee$ denotes taking the component-wise maximum.
\end{prop}
Multivariate regular variation provides an asymptotic framework and for given $\alpha, \{a_n\}$ and $\mu$ there exist several distributions of random vectors $Y$ such that $Y \in {\rm RV}_\alpha(\{a_n\},\mu)$, but according to Proposition~\ref{p:RV:via:maxima} their maxima are all attracted to the same random vector $X$ whose distribution depends only on $\mu$. The class of limiting random variables in Proposition~\ref{p:RV:via:maxima} will be inspected more closely in the next section.

\subsection{Max-stable vectors} 
The homogeneity property \eqref{e:mu-scaling} of $\mu$ implies that the limiting random vector in Proposition~\ref{p:RV:via:maxima} has a certain stability property, namely that
\begin{equation}\label{e:general-max-stability-def}
\bigvee_{t=1}^{n}X^{(t)}  \stackrel{d}{=} n^{1/\alpha} X \;\;\; \mbox{ for all } n \in \mathbb{N},
\end{equation}
with the same notation as in Proposition~\ref{p:RV:via:maxima} and where $\stackrel{d}{=}$ stands for equality in distribution, see \cite{kulik:soulier:2020}, Section 2.1. We call such a random vector $X$ \emph{max-stable} and we call $X$ \emph{non-degenerate max-stable} if in addition $\P[X=(0, \ldots, 0)]<1$. For $\alpha=1$ this simplifies to 
\begin{equation}\label{e:max-stability-def}
\bigvee_{t=1}^{n}X^{(t)}  \stackrel{d}{=} n X \;\;\; \mbox{ for all } n \in \mathbb{N},
\end{equation}
and we speak of a \emph{simple max-stable random vector} $X$, which we will further analyze in the following. 

The marginal distributions of simple max-stable distributions are necessarily $1$-Fr\'echet, that is,
$$
\P[ X_i \le x ] = e^{-\sigma_i/x},\ x>0,
$$
for some non-negative scale coefficient $\sigma_i$.  We shall write $\|X_i\|_1:=\sigma_i$ for the scale coefficient of the 
$1$-Fr\'echet variable $X_i$. The next result characterizes all multivariate simple max-stable distributions. Here, we recall the so-called de Haan construction of a simple max-stable vector. 

\begin{prop}\label{p:de-Haan} Let $(E,{\cal E},\nu)$ be a measure space and let $L_+^1(E,\nu)$ denote the set of all non-negative 
$\nu$-integrable functions on $E$. For every collection $f_i\in L_+^1 (E,\nu),\ 1\leq i\leq p$, there is a 
random vector $X = (X_i)_{1\leq i\leq p}$, such that for all $x_i>0, 1 \leq i \leq p,$
\begin{equation}\label{e:de-Haan-fdd}
\P[ X_i\le x_i,\ 1\leq i\leq p] = \exp\Big\{-\int_{E} \max_{1\leq i\leq p} \frac{f_i(u)}{x_i} \nu(du) \Big\}.
\end{equation}
The random vector $X$ is simple max-stable.  Conversely, for every simple max-stable vector $X$, Equation \eqref{e:de-Haan-fdd} holds and $(E,{\cal E},\nu)$ can be chosen as $([0,1],{\cal B}[0,1],{\rm Leb})$.  In fact, we have the stochastic representation
\begin{equation}\label{e:de-Haan}
(X_i)_{1\leq i\leq p} \stackrel{d}{=} (I(f_i))_{1\leq i\leq p},\ \ \mbox{ with }\ I(f):= \bigvee_{j=1}^\infty \frac{f(U_j)}{\Gamma_j},\ f\in L_+^1([0,1],\nu),
\end{equation}
where $\{(\Gamma_j,U_j)\}$ is a Poisson point process on $(0,\infty)\times [0,1]$ with mean measure $dx\times \nu(du)$.
\end{prop}

For a proof and more details, see e.g. \cite{dehaan:1984,stoev:taqqu:2005}.  The functions $f_i$ in \eqref{e:de-Haan-fdd} and \eqref{e:de-Haan} are referred to as {\em spectral functions} associated with the vector $X$.  From \eqref{e:de-Haan-fdd} and \eqref{e:de-Haan}, one can readily see that for all $f\in L_+^1(E,\nu)$,
the so-called {\em extremal integral} $I(f)$ in \eqref{e:de-Haan} is a well-defined $1$-Fr\'echet random variable.  More precisely, its cumulative distribution function is:
$$
\P[ I(f)\le x] = \exp\{ - \|I(f)\|_1/x\},\; x>0, \ \ \mbox{ where }\|I(f)\|_1 = \| f \|_{L^1}= \int_{E} f(u) \nu(du).
$$
Moreover, the extremal integral functional $I(\cdot)$ is {\em max-linear} in the sense that for all $a_i\ge 0$ and 
$f_i\in L_+^1(E,\nu),\ 1\leq i\leq n$, we have
$$
I\Big(\bigvee_{t=1}^{n} a_t f_t\Big) = \bigvee_{t=1}^{n} a_t I(f_t).
$$
Thus, every max-linear combination $\vee_{i=1}^n a_i X_i$ of $X$ as above with coefficients $a_i\ge 0$ 
is a $1$-Fr\'echet random variable with scale coefficient:
$$
\Big\| \bigvee_{i=1}^n a_i X_i \Big\|_1 = \int_{E} \Big( \bigvee_{i=1}^n a_i f_i(u) \Big) \nu(du) = \Big\|\bigvee_{i=1}^n a_i f_i\Big\|_{L^1}.
$$
We will further explore the asymptotic properties of simple max-stable random vectors and how they fit into the framework of multivariate regular variation in the following section. 
\subsection{Extremal dependence functionals and tail-dependence coefficients} 
The tail measure $\mu$ and the normalizing sequence $\{a_n\}$ from Section \ref{subsec:MRV} provide a comprehensive description of the asymptotic behavior of a random vector $X$ and allow to approximate probabilities of the form $\P[X \in a_n A]$ for all sets $A$ bounded away from 0. Sometimes, however, one may be interested in those probabilities for certain simple sets $A$ only and describe the asymptotic behavior of $X$ by certain extremal dependence functions instead. In this section, we first derive a general result for such extremal dependence functions and then introduce two particularly popular families of them.

\begin{prop}\label{p:h(Y)} Let $X\in RV(\{a_n\},\mu)$ in $\R^p$ with index $\alpha>0$.  
Let also $h:\R^p\to [0,\infty)$ be a non-negative, continuous and $1$-homogeneous function, i.e., 
$h(c x) = c h(x),\ c>0,\ x\in \R^p$.  Then,
\begin{equation}\label{e:p:h(Y)}
\lim_{n\to\infty} n\P[ h(X) > a_n] = \int_S h(u)^\alpha \sigma(du) = \E[h(Y)^\alpha]\sigma(S),
\end{equation}
where $Y$ has probability distribution $\sigma(\cdot)/\sigma(S)$ and $\sigma$ is as in \eqref{e:sigma-spectral}.
\end{prop}

Though this result is similar to \cite{yuen:2020}, Lemma A.7, and also a special case to
\cite{Dyszewski2020}, Theorem 2.1, its proof is given Section \ref{sec:proofs}.

We will apply the formula in \eqref{e:p:h(Y)} for homogeneous functionals 
of the form $h(x) = (\min_{i\in K} x_i)_+$ and $h(x) = (\max_{i\in K} x_i)_+$ 
for some subset $K\subset [p]=\{1,\ldots,p\}$.

The next result shows that simple max-stable vectors are regularly varying 
and provides means to express their extremal dependence functionals \textit{both} in terms of spectral functions \textit{and} tail measures.

\begin{prop}\label{p:X-max-stable-is-RV}  Let $X=(X_i)_{1\leq i\leq p}$ be a non-degenerate simple max-stable vector as in \eqref{e:de-Haan-fdd}.  Then, $X\in {\rm RV}_1(\{n\},\mu)$, where
$\mu$ is supported on $[0,\infty)^p$ and for all $x=(x_i)_{1\leq i\leq p}\in \R_+^p\setminus\{0\}$
$$
\P[ X\le x] = \exp\{-\mu([0,x]^c) \}, \;\; \mbox{ with }  \mu([0,x]^c)=\int_{E} \max_{1\leq i\leq p} \frac{f_i(u)}{x_i} \nu(du).
$$ 
Moreover, for every non-negative, continuous $1$-homogeneous function $h:\R^p\to[0,\infty)$, we have
\begin{equation}\label{e:p:max-stable-RV}
 \lim_{n\to\infty} n \P[h(X)>n] = \mu(\{h>1\})  = \int_S h(u)\sigma(du)=
 \int_{E} h(\vec{f}(z))\nu(dz),
\end{equation}
where $\vec{f}(z) = (f_1(z),\cdots,f_p(z))$. In particular, the spectral 
measure $\sigma$ has the representation
\begin{equation}\label{e:p:max-stable-RV-sigma}
\sigma(B)= \int_{E} 1_B\Big(\frac{\vec{f}(z)}{\|\vec{f}(z)\|}\Big) \|\vec{f}(z)\| \nu(dz),\ \ 
B\in {\cal B}(S).
\end{equation}
\end{prop}
Again, this result is standard but we sketch its proof for the sake of completeness 
in Appendix \ref{sec:proofs}.  The classic representation of the simple max-stable cumulative distribution functions is a 
simple corollary from  Proposition \ref{p:X-max-stable-is-RV}.

\begin{cor} In the situation of Proposition \ref{p:X-max-stable-is-RV}, by taking $h(u):= h_x(u) := (\max_{i\in [p]} u_i/x_i)_+$ for $x_i\in (0,\infty)^p$ in \eqref{e:p:max-stable-RV}, we obtain 
$\mu(\{h>1\}) = \mu([0,x]^c)$ and 
\begin{align}\label{e:max-stable-CDF-via-sigma}
 \P[ X\le x] = \exp\{-\mu([0,x]^c)\} = \exp\Big\{ -\int_S \Big( \max_{i\in [p]} \frac{u_i}{x_i}\Big) \sigma(du)\Big\}.
\end{align}
\end{cor}

For more details on the characterization of the max-domain 
of attraction of multivariate max-stable laws in terms of  multivariate regular variation, 
see e.g., Proposition 5.17 in \cite{resnick:1987}.

We are now ready to recall the general definitions of the extremal and tail-dependence coefficients 
of a regularly varying random vector, which have briefly been introduced in Section \ref{sec:Introduction}, now with additional notation for the normalizing sequence $\{a_n\}$. 

\begin{definition}\label{def:tail-dep-coeff} Let $X=(X_i)_{1\leq i\leq p} \in {\rm RV}(\{a_n\},\mu)$.
Then, for non-empty sets $K, L\subset[p]$, we let
$$
\theta_X(K;\{a_n\}):= \lim_{n\to\infty} n\P\Big[\max_{i\in K} X_i>a_n\Big]\ \ \mbox{ and }\ \ 
 \lambda_X(L;\{a_n\}) := \lim_{n\to\infty} n\P\Big[\min_{i\in L} X_i>a_n\Big].
$$
The $\theta_X(K;\{a_n\})$'s and $\lambda_X(L;\{a_n\})$'s are referred to as the 
\textit{extremal and tail-dependence coefficients relative to $\{a_n\}$} of the vector $X$, respectively.  
\end{definition}

If it is clear to which random vector we refer to or it does not matter for the argument, we may drop the index $X$ and just write $\theta(K;\{a_n\})$ and $\lambda(K;\{a_n\})$. Sometimes we will view $\theta$ and $\lambda$ as functions of $k$-tuples and write for example
$$ \lambda_X(i_1,\cdots,i_k;\{a_n\}) ,\ 1\le i_1,\ldots,i_k\le p,
$$
(where some of the arguments $i_1,\ldots,i_k$ may repeat) which corresponds to $\lambda_X(L,\{a_n\})$ where $L$ 
is the set of all distinct values in $\{i_1,\ldots,i_k\}$.

\begin{remark}
Note that the definitions of $\theta_X(K,\{a_n\})$ and $\lambda_X(L,\{a_n\})$ depend on the choice of the 
sequence $\{a_n\}$. They are unique, however, up to a multiplicative constant.  More precisely,
if ${\rm index}(X) = \alpha$ and $a_n\sim a_n', c>0$, then
$$
\theta_X(K;\{c a_n\}) = c^{-\alpha} \theta_X(K;\{a_n'\})\ \ \mbox{ as well as } \ \ 
\lambda_X(L;\{c a_n\}) = c^{-\alpha} \lambda_X(L;\{a_n'\}).
$$
\end{remark}

\begin{remark}
In the following we will focus on extremal and tail-dependence coefficients of max-stable random vectors, which exist by Definition \ref{def:tail-dep-coeff} in combination with Proposition \ref{p:X-max-stable-is-RV} as long as $X$ is non-degenerate. Observe that if $X$ is non-degenerate simple max-stable, then
$$
 \lambda(i;\{n\}) = \theta (i;\{n\}) = \lim_{n \to \infty} n \P[X_i>n]=\lim_{n \to \infty} n(1-e^{-\sigma_i /n})=\sigma_i=\|X_i\|_1,\ 1\leq i\leq p.
$$
Thus, if all marginals of $X$ are standard $1-$Fr\'echet, i.e., $\|X_i\|_1=1$, then setting $a_n=n$ ensures that $\lim_{n \to \infty} n \P[X_i>a_n]=1$ and one recovers the upper tail-dependence coefficient $\lambda_X(i,j)$ from \eqref{e:lambda-bivariate-classic}, $i,j \in [p]$. 
More generally, if $X$ is non-degenerate simple max-stable, then we can choose $a_n=n$ as a normalizing sequence and in this case (or if the sequence $\{a_n\}$ does not matter for the argument), we will also write
   $$ \theta(K)=\theta_X(K)=\theta_X(K;\{a_n\}),\, \;\;\; \lambda(L)=\lambda_X(L)=\lambda_X(L;\{a_n\}), \;\;\; K, L \subset [p]. $$
In the case that $\P[X=(0,\ldots, 0)]=1$, we set $\theta_X(K)=\lambda_X(L)=0$ for all $K, L \subset [p]$. 
\end{remark}

The following result expresses these functionals in terms of both the tail measure
$\mu$ and the spectral functions of the vector $X$. Again, the proof is given in Appendix \ref{sec:proofs}.

\begin{cor}\label{c:theta-lambda-mu} Let $X=(X_i)_{1\leq i\leq p}$ be a simple max-stable vector as in \eqref{e:de-Haan-fdd}. Then,
\begin{equation}\label{e:theta-lambda-mu}
\theta_X(K) = \mu \Big( \bigcup_{i\in K} A_i \Big)\ \mbox{ and }\ \lambda_X(L)
= \mu \Big( \bigcap_{i\in L} A_i \Big),
\end{equation}
where $A_i := \{ x\in \R^p\, :\, x_i>1\}$ and
\begin{equation}\label{e:theta-lambda-f}
\theta_X(K) =\int_E \max_{i\in K} f_i(x) \nu(dx)  \ \mbox{ and }\ \lambda_X(L) = \int_E \min_{i\in L} f_i(x) \nu(dx).
\end{equation}
\end{cor}

\subsection{Bivariate tail-dependence measures and spectral distance}\label{sec:bivariate-intro}
In Definition \ref{def:tail-dep-coeff} we introduced general extremal and tail-dependence coefficients for arbitrary non-empty subsets $K,L\subset [p]$, i.e.\ for $2^p-1$ different sets. Often these are too many coefficients for a handy description of the dependence structure. Therefore, one may consider only the pairwise dependence in a simple max-stable vector $X$ which corresponds to the consideration of sets $K$ and $L$ with at most two entries. The set of tail-dependence coefficients with sets containing at most two elements can be written in the so called \emph{matrix of bivariate tail-dependence coefficients}, which we denote by $$\Lambda_X = \Lambda = (\lambda_X(i,j))_{1\leq i,j\leq p} = (\lambda_X(i,j;\{n\}))_{1\leq i,j\leq p}.$$  
For the bivariate tail-dependence we have the alternative representation
\begin{align}
	\lambda_X(i,j)&=\lim_{n\to\infty}n\P[ X_i >n, X_j>n ]
	 =\lim_{n\to\infty}n(\P[ X_i >n]+\P[ X_j>n ]-\P[ X_i \vee  X_j>n ]) \\
	&=\|X_i\|_1+\|X_j\|_1-\|X_i\vee X_j\|_1.\label{eq:rep_lampda}
\end{align} 
For standardized marginals $\|X_i\|=1$ this implies $\lambda_X(i,j)=2-\|X_i\vee X_j\|_1$. 
The $1$-Fr\'echet marginals of $X$ imply 
$$\P[X_i>n]\sim \frac{\|X_i\|_1}{n}\,\,\,\,\text{ and } \,\, \P[X_i\vee X_j>n]\sim \frac{\|X_i\vee X_j\|_1}{n} $$ 
as $n\to\infty$, where $\|X_i\vee X_j\|_1$ denotes the scale coefficient of the $1$-Fr\'echet distribution of $X_i\vee X_j$. 
Thus, for standardized marginals $\|X_i\|_1=1$, $1\leq i\leq p$, the bivariate tail-dependence coefficients also have the following  representation for all $1\leq i,j\leq p$:
\begin{align}\label{eq:lambda-conditional}
	\lambda_X({i,j})=\lim_{u\to\infty} \P[ X_i >u, X_j>u ]/\P[X_i>u]
	=\lim_{n\to\infty}\P[ X_j >n\mid  X_i>n ].\quad
\end{align}

In this form, the bivariate tail-dependence matrix is a popular measure for the extremal dependence in the random vector $X$. First appearing around the 60's (e.g.\ \cite{tiago:1962}), the bivariate tail-dependence coefficients are frequently considered in the literature, see e.g. \cite{Coles:1999,beirlant:2004,frahm:2005,fiebig:2017,Shyamalkumar2020} for different considerations (sometimes other names as \textit{coefficient of (upper) tail-dependence} or \textit{$\chi$-measure} are used). In the context of finance and insurance but also in an environmental context this measure is used to describe the extremal risk in the random vector $X$. Moreover, the characterization of whether $X_i$ and $X_j$ are extremally dependent is usually formulated by these bivariate tail-dependence coefficents: If $\lambda_X(i,j)=0$, then $X_i$ and $X_j$ are extremally independent, otherwise the two random variables are extremally dependent.

Note that for standardized marginals the relation $\theta_X(i,j)=2-\lambda_X(i,j)$ holds. The extremal dependence coefficient in this form has often been used in the literature as a measure for extremal dependence, see e.g. \cite{smith:1990,schlather:tawn:2003,Strokorb:2015-extremal}.

In all these references, the tail-dependence coefficient was defined as in \eqref{eq:lambda-conditional} and standardized (or at least identically distributed) marginal distributions were assumed, as it is common for the analysis of dependence. However, we allow for unequal scales and therefore use the more general form \eqref{eq:rep_lampda}.

\begin{remark} The matrix of bivariate tail-dependence coefficients $\Lambda$ of a simple max-stable
vector is necessarily positive semi-definite.  Indeed, this follows from the observation that by Corollary~\ref{c:theta-lambda-mu}
$$
\lambda(i,j) = \int f_i(x) \wedge f_j(x) \nu(dx) = \int {\rm Cov}(B(f_i(x)), B(f_j(x))) \nu(dx),
$$
where $B=\{B(t),\ t\ge 0\}$ is a standard Brownian motion and since non-negative 
mixtures of covariance matrices are again covariance matrices. 
Another way to see this is from the observation that for each $n$, we have
$n\P[X_i>n, X_j>n] =n\E[ I(X_i>n) I(X_j>n)]$ is a positive semi-definite function of $i,j\in [p]$,
which is related to the fact that $(i,j)\mapsto \lambda(i,j)$ is, up to a multiplicative
constant, the covariance function of a certain random exceedance set (see Remark \ref{e:lambda-ij-is-set-covariance}, below).
\end{remark}

The matrix $\Lambda$ is thus positive semi-definite, has non-negative entries and for standardized marginals of $X$ it holds $\lambda(\{i\})=1$, i.e. $\Lambda$ is a correlation matrix. However, not every correlation matrix with non-negative entries is necessarily a matrix of bivariate tail-dependence coefficients. The realization problem (i.e. the question whether a given matrix is the tail-dependence matrix of some random vector) is a recent topic in the literature \citep{fiebig:2017,krauseetal2018,Shyamalkumar2020}. We will further discuss this problem in Section \ref{sec:NP}.

Related to the bivariate dependence coefficients we define an associated function, which will turn out to be a semi-metric on $[p]$. 
\begin{definition}\label{def:specdist}
Let $X=(X_i)_{1\leq i\leq p}$ be a simple max-stable vector. Then, for $i,j\in [p]$, the \emph{spectral distance} $d_X$ is defined by 
\begin{equation}
	\label{eq:dist} d_X(i,j):=d(X_i,X_j):=2\|X_i\vee X_j\|_1-\|X_i\|_1-\|X_j\|_1.
\end{equation} 
\end{definition}

By \eqref{eq:rep_lampda} 
\begin{align}
	\label{eq:d-lambda}	d_X(i,j)&=\|X_i\|_1+\|X_j\|_1-2\lambda_X(i,j)
	=\lambda_X(i)+\lambda_X(j)-2\lambda_X(i,j).
\end{align}
If the scales of the marginals of the simple max-stable vector $(X_i)_{1\leq i\leq p}$ are the same, i.e. $\|X_i\|=c>0$ for some $c>0$ and all $1\leq i\leq p$, then \eqref{eq:d-lambda} simplifies to 
\begin{equation}\label{eq:d-lambda-simple}
    d(i,j)=2(c-\lambda_X(i,j)).
\end{equation}
For standard $1$-Fr\'echet marginals this further reduces to $d(i,j)=2(1-\lambda_X(i,j))$.

The spectral distance for max-stable vectors was already considered in \cite{stoev:taqqu:2005}, equation (2.11). There it was shown that this distance is indeed a semi-metric on $[p]$ \citep[Proposition 2.6]{stoev:taqqu:2005} and that it metricizes convergence in probability in $1$-Fr\'echet spaces \citep[Proposition 2.4]{stoev:taqqu:2005}. 
In the form of \eqref{eq:d-lambda}, the spectral distance also appears in \cite{fiebig:2017}, where it was defined in two steps in \citep[Proposition 34 and 37]{fiebig:2017}. There, the use of the spectral distance is based on the fundamental work of \citep[Section 5.2]{deza:1997}, where it is used in a different context. 

In Section \ref{sec:l1-emb} we will prove that the spectral distance of a simple max-linear vector $X$ is $L^1$-embeddable, with representation $d_X(i,j)=\|f_i-f_j\|_{L^1}$, where $f_i,f_j$ are the spectral functions of $X$. In this form, the spectral distance was already used in \cite{davis:1989, davis:1993}, where it was mainly applied for a projection method for prediction of max-stable processes. \cite{davis:1993} also gave a connection to the bivariate tail-dependence coefficients $\lambda(i,j)$ as considered in \cite{tiago:1962}, but only in the case of equally scaled marginals.

\section{Tail-dependence via exceedance sets}\label{sec:TM}

In this section we develop a unified approach to representing tail-dependence via random exceedence sets, which explains and extends the notion of Bernoulli compatibility discovered in \cite{embrechts:hofert:wang:2016} to higher order tail-dependence. Moreover, we introduce a slight extension of the so-called Tawn-Molchanov models and explore their connections to extremal and 
tail-dependence coefficients. 

\subsection{Bernoulli compatibility}  \label{sec:exceedance-set}

We will first demonstrate that tail-dependence can be succinctly
characterized via a random set obtained as the limit of exceedance sets.  
Let $X \in {\rm RV}_\alpha(\{a_n\},\mu)$ and consider the {\em exceedance set}:
$$
\Theta_n:= \{ i\, :\, X_i>a_n\}.
$$
The asymptotic distribution of this random set, conditioned on it being non-empty can
be directly characterized in terms of the extremal or tail-dependence coefficients of $X$.
Specifically, these dependence coefficients can be seen as the {\em hitting} and {\em inclusion}
functionals of a \emph{limiting random set} $\Theta$, respectively.  For the 
precise definitions and related notions from the theory of random sets, we will always refer to the monograph of \cite{molchanov:2017}.\\

Before proceeding with the analysis of $\Theta$ we will introduce some appropriate coefficients. Let
\begin{equation}\label{e:beta-via-mu}
\beta(J):= \mu(B_J):= \mu\Big( \bigcap_{j\in J} A_j \cap\bigcap_{k\in J^c} A_k^c \Big),\ \ \emptyset\not=J\subset [p],
\end{equation}
where again $A_i := \{ x\in \R^p\, :\, x_i>1\}, i \in [p]$.
Then, in view of \eqref{e:theta-lambda-mu}, since the $B_J$'s are all pairwise disjoint in 
$J$,
\begin{equation}\label{e:theta-lambda-via-beta}
\theta_X(K) = \sum_{J\, :\, J\cap K \not = \emptyset } \beta(J)\ \ \mbox{ and }\ \ \lambda_X(L) 
= \sum_{J\, :\, L\subset J} \beta(J).
\end{equation}
This, in view of the so-called M\"obius inversion formula, see, e.g., \cite{molchanov:2017}, Theorem 1.1.61, 
yields the inversion formulae:
\begin{equation}\label{e:beta-via-theta}
\beta(J) = \sum_{K\, :\, \emptyset \not=K,\ J^c\subset K} (-1)^{|J\cap K|+1} \theta_X(K) ,
\end{equation}
which is Equation (7) in \cite{schlather:tawn:2003}, Theorem 1.
We also have
\begin{equation}\label{e:beta-via-lambda}
\beta(J) = \sum_{L\, :\, J \subset L \subset [p]} (-1)^{|L\setminus J|} \lambda_X(L).
\end{equation}
Finally, the usual {\em inclusion-exclusion} type relationships hold
between $\theta$ and $\lambda$:
\begin{equation}\label{e:theta-and-lambda}
\theta_X(K) = \sum_{L\, :\, \emptyset\not=L\subset K}(-1)^{|L|-1}\lambda_X(L)\ \ \mbox{ and }\ \ 
\lambda_X(L) = \sum_{K\, :\, \emptyset \not=K\subset L} (-1)^{|K|-1} \theta_X(K). 
\end{equation}
Although some of the Relations \eqref{e:beta-via-theta}, \eqref{e:beta-via-lambda}, and
\eqref{e:theta-and-lambda} are available in the literature, we prove them in Appendix \ref{sec:proofs}
independently with elementary arguments in Lemma \ref{l:formulae}. \\

Observe that the event $\{\Theta_n \cap K \not = \emptyset\}$ is $\{\max_{i\in K} X_i >a_n\}$.  This immediately implies that
$$
T_n(K):=\P[ \Theta_n \cap K \not=\emptyset \ |\ \Theta_n\not=\emptyset] 
= \frac{\P[ \max_{i\in K} X_i >a_n]}
{\P[\max_{i\in[p]} X_i>a_n]} \longrightarrow \frac{\theta_X(K)}{\theta_X([p])},\ \mbox{ as $n\to\infty$}.
$$
The functionals $T_n(\cdot)$ are known as the hitting functionals of the conditional distribution of the random set $\Theta_n$.  
They are completely alternating capacities and their limit yields hitting functionals $T(K):= \theta_X(K)/\theta_X([p])$ of a non-empty random set $\Theta \subset [p]$.
This random set $\Theta$ may be viewed as the ``typical'' exceedance set for a regularly varying vector as the  threshold $a_n$ approaches infinity. 
It is immediate from \eqref{e:beta-via-theta} and \cite{molchanov:2017}, Corollary 1.1.31, that
\begin{equation}\label{e:Theta-PMF}
\P[\Theta = J] = \frac{\beta(J)}{\sum_{\emptyset \neq K \subset [p]} \beta(K)},\ \ \emptyset\not=J\subset[p].
\end{equation}
Observing that $\theta_X([p]) = \sum_{\emptyset \neq K \subset [p]} \beta(K),$ we have thus established the following result.
\begin{prop}\label{p:Theta-set} Let $X \in {\rm RV}(\{a_n\},\mu)$ and define the random exceedance set
$\Theta_n:= \{i\, :\, X_i>a_n\}$.  Then, as $n\to\infty$, we have
$$
\P[\Theta_n \in \cdot | \{\Theta_n\not = \emptyset\}] \Rightarrow \P[\Theta \in \cdot],
$$
where the probability mass function of $\Theta$ is as in \eqref{e:Theta-PMF} and
the $\beta(J)$'s are as in \eqref{e:beta-via-mu}.  We have moreover that 
\begin{equation}\label{e:lambda-via-the-set-Theta}
\P[\Theta \cap K \not =\emptyset] = \frac{\theta_X(K)}{\theta_X([p])}\ \ \mbox{ and }\ \ 
\P[L\subset \Theta] = \frac{\lambda_X(L)}{\theta_X([p])}.
\end{equation}
\end{prop}

\begin{remark}
\cite{molchanov:strokorb:2016} introduced the important class of
Choquet random sup-measures whose distribution is characterized by the extremal 
coefficient functional $\theta(\cdot)$.  This is closely related 
but not identical to our perspective here, which emphasizes
threshold-exceedance rather than max-stability.
\end{remark}

The above result shows that all tail-dependence coefficients can be succinctly represented (up to a constant) via the random set $\Theta$. This finding allows us to connect the tail-dependence coefficients to so-called {\em Bernoulli-compatible} tensors.  

\begin{definition} A $k$-tensor $T = (T(i_1,\cdots,i_k))_{1\le i_1,\cdots,i_k\le p}$ is said to be Bernoulli-compatible, if 
\begin{equation}\label{e:T-Bernoulli-compatible}
T(i_1,\cdots,i_k) = \E \Big[ \xi(i_1)\cdots \xi(i_k) \Big],
\end{equation}
where $\xi(1),\cdots,\xi(p)$ are (possibly dependent) Bernoulli $0$ or $1$-valued random variables, i.e. $P(\xi(i)=1)=p_i=1-P(\xi(i)=0)$ for some $p_i \in [0,1], i \in [p]$. If not all $\xi(i)$'s are 
identically zero, the tensor $T$ is said to be non-degenerate.
\end{definition}

In the case $k=2$, this definition recovers the notion of Bernoulli compatibility in \cite{embrechts:hofert:wang:2016}. Proposition \ref{p:Theta-set} implies the following result.

\begin{prop}\label{p:Bernoulli-arrays}\phantom{x}
\begin{itemize}
    \item[(i)] For every Bernoulli-compatible $k$-tensor 
    $T = (T(i_1,\cdots,i_k))_{[p]^k}$, there exists a 
    simple max-stable random vector $X$ such that
$$
T(i_1,\cdots,i_k) = \lambda_X(i_1,\cdots,i_k),
$$
for all $i_1,\cdots,i_k\in [p]$.
\item [(ii)] Conversely, for every simple max-stable random vector $X=(X_i)_{1\leq i\leq p}$, and every 
$c\ge \theta_X([p])$
\begin{equation}\label{e:p:Bernoulli-arrays-ii} 
(T(i_1,\cdots,i_k))_{[p]^k} := \frac{1}{c}\cdot \Big( \lambda_X(i_1,\cdots,i_k)\Big)_{[p]^k}
\end{equation}
is a Bernoulli-compatible $k$-tensor.
\end{itemize}
\end{prop}

\begin{proof} $(i):$ Assume \eqref{e:T-Bernoulli-compatible} holds and 
introduce the random set $\Theta :=\{i\, :\, \xi(i)=1\}$. Let $\beta(J) := \P[\Theta =J]$ and define
the simple max-stable vector 
\begin{equation}\label{eq:TM-in-proof}
X:= \bigvee_{J\,:\, \emptyset\not=J\subset [p]} \beta(J) 1_J Z_J,
\end{equation}
where $1_{J}=(1_J(i))_{1\leq i\leq p}$ contains $1$ in the coordinates in $J$ and $0$ otherwise and the $Z_J$'s are iid standard $1$-Fr\'echet. In view of Lemma \ref{lemma:d-linear} and since $\lambda_{1_J Z_J}(L)=1$ for $L\subset J$ and $\lambda_{1_J Z_J}(L)=0$ for $L\not\subset J$, we have 
$$
\lambda_X(L) = \sum_{J\, :\, L\subset J} \beta(J) = \P[ L\subset \Theta].
$$
Since for $L = \{i_1,\cdots,i_k\}$ we have $1_{\{L\subset \Theta\}} = \prod_{j=1}^k \xi(i_j)$, we obtain
$$
T(i_1,\cdots,i_k) = \E[ \xi(i_1)\cdots\xi(i_k)]=\P[L\subset \Theta] = \lambda_X(L).
$$
This completes the proof of (i).\\

$(ii):$ If $\P[X=(0, \ldots, 0)]=1$, then $\theta_X([p])=0$ and the statement follows by setting all $\xi(i_k)$ identically to 0, so assume $\P[X = (0, \ldots, 0)]<1$ in the following, which implies $\theta_X([p])>0$.  Let $\Theta\subset [p]$ be a 
random set such that \eqref{e:lambda-via-the-set-Theta} holds, i.e.,
$$
\lambda_X(L) = \theta_X([p])\cdot \P[L\subset \Theta],\ L\subset [p].
$$ 
Define $\xi(i):= B\cdot 1_{\Theta}(i)$, where $B$ is a Bernoulli random variable, independent of $\Theta$, such that
$\P[B=1] = 1-\P[B=0] = q\in (0,1]$ for all $i \in [p]$.  Then, we have that
$$
\E[ \xi(i_1)\cdots \xi(i_k)] = q \E[ 1_{\{i_1,\cdots,i_k\} \subset \Theta}] = \frac{q}{\theta_X([p])} \cdot
\lambda_X(i_1,\cdots,i_k).
$$
This shows that \eqref{e:p:Bernoulli-arrays-ii} holds with potentially any $c\ge \theta_X([p])$. 
\end{proof}

\begin{remark} As it can be seen from the proof the lower bound on the 
constant $c$ in Proposition \ref{p:Bernoulli-arrays} (ii) cannot be improved. 
Observe that $\theta([p]) \le \sum_{i\in [p]} \lambda_X(i)$,
where the inequality is strict unless all $X_i$'s are independent. 
Thus, the above result even in the case $k=2$ improves upon Theorem 3 
in \cite{krauseetal2018} where the range for the constant $c$ is $c\ge \sum_{i\in [p]} \lambda_X(i)$.
\end{remark}

\begin{remark}\label{e:lambda-ij-is-set-covariance}
In the case of two-point sets, we have that the bivariate tail-dependence coefficient
\begin{equation}\label{e:lambda-via-Theta}
\lambda(i,j) = \theta([p]) \times \P[ i,j\in \Theta] = \theta([p]) \E[1_{\Theta}(i) 1_{\Theta}(j)] ,\ \ i,j\in[p],
\end{equation}
is proportional to the so-called covariance function 
$(i,j)\mapsto \P[i,j\in \Theta] =\E[1_{\Theta}(i) 1_{\Theta}(j)]$
of the random set $\Theta$.  This shows again that the bivariate tail-dependence function 
$(i,j)\mapsto \lambda(i,j)$ is positive semidefinite.
\end{remark}

\begin{remark}\label{rem:Bercomp} Relation \eqref{e:lambda-via-Theta} recovers a
simple proof of the Bernoulli compatibility of TD
matrices established in Theorem 3.3 of \cite{embrechts:hofert:wang:2016}. 
Namely, their result states that $\Lambda = (\lambda_{i,j})_{p\times p}$
is a matrix of bivariate tail-dependence coefficients, if and only if 
$\Lambda = c \E [\xi \xi^\top]$ for some $c>0$ and a random vector
$\xi=(\xi_i)_{1\leq i\leq p}$ with Bernoulli entries taking values in $\{0,1\}$.
Clearly, there is a one-to-one correspondence between a random set $\Theta\subset [p]$ and a Bernoulli random vector: $\Theta:=\{i\, :\, \xi_i=1\}$ and 
$\xi = (1_{\Theta}(i))_{1\leq i\leq p}$.  The characterization 
result then follows from \eqref{e:lambda-via-Theta}.
\end{remark}

\subsection{Generalized Tawn-Molchanov models}

In the previous section we defined in \eqref{e:beta-via-mu} coefficients $\beta(J)$ to characterize the distribution of the limiting exceedance set $\Theta$. These coefficients were then used in \eqref{eq:TM-in-proof} to construct a max-stable random vector in order to prove Proposition~\ref{p:Bernoulli-arrays}. This special  random vector is in fact nothing else than a generalized version of the so-called Tawn-Molchanov model which we will introduce formally in this section. 

The following result is a slight extension and re-formulation of existing results in the literature, which have first appeared in \cite{Schlather:2002,schlather:tawn:2003} (see also
\cite{Strokorb:2015-extremal,molchanov:strokorb:2016} for extensions) in the context of finding necessary and sufficient
conditions for a set of $2^p-1$ numbers $\{\theta(K)\mid \, \emptyset \not=K\subset[p]\}$ to be the extremal coefficients 
of a max-stable vector $X$. The novelty here is that we consider max-stable vectors with possibly non-identical 
marginals and treat simultaneously the cases of extremal as well as tail-dependence coefficients.

\begin{theorem}\label{t:TM-model} The function $\{\theta(K),\ K
\subset[p]\}$ ($\{\lambda(L),\ L\subset[p]\}$, respectively) 
yields the extremal (tail-dependence, respectively) coefficients of a 
simple max-stable vector $X=(X_i)_{1\leq i\leq p}$ if and only if the $\beta(J)$'s in \eqref{e:beta-via-theta} (\eqref{e:beta-via-lambda}, respectively) are non-negative for all 
$\emptyset\not=J\subset [p]$.  In this case, let $Z_J, J \subset [p]$, be iid standard $1$-Fr\'echet random variables and define
\begin{equation}\label{e:TM-model}
X^* = (X_i^*)_{1\leq i\leq p} := \bigvee_{\emptyset \neq J \subset [p]} \beta(J) 1_J Z_J,
\end{equation}
where $1_{J}=(1_J(i))_{1\leq i\leq p}$ contains $1$ in the coordinates in $J$ and $0$ 
otherwise.  Then, $X^*$ is a max-stable random vector whose extremal (tail-dependence) coefficients are precisely the $\theta(K)$'s ($\lambda(L)$'s, respectively).
\end{theorem}

The proof is given in Appendix \ref{sec:proofs}.  The vector $X^*$ defined in \eqref{e:TM-model} is referred to as the
Tawn-Molchanov or simply \emph{TM-model} associated with the extremal (tail-dependence) coefficients $\{\theta(K)\}$
($\{\lambda(L)\},$ respectively).

\begin{remark} The distribution of the random set $\Theta$ introduced in Section \ref{sec:exceedance-set} can be 
understood in terms of the Tawn-Molchanov model \eqref{e:TM-model} using the {\em single large jump} heuristic.  Given 
that  $\Theta_n = \{i\, :\, X^*_i>n\} \not=\emptyset$, for large $n$, only one of the $Z_J$'s is extreme enough to
contribute to the exceedance set.  Thus, with high probability, $\Theta_n$ equals the corresponding $J$ in 
\eqref{e:TM-model}.  The probability of the set $J$ to occur is asymptotically proportional to the weight $\beta(J)$,
which explains the formula \eqref{e:Theta-PMF}.
\end{remark}

We have seen in Section \ref{sec:bivariate-intro} that extremal dependence can also be measured in terms of spectral distance. In the following section we will explore further the connections between spectral distance and the just introduced Tawn-Molchanov models and see how the latter naturally lead to a decomposition of the former which is equivalent to $\ell_1$-embeddability.

\section{Embeddability and rigidity of the spectral distance}\label{sec:l1-emb}

So far, we have mainly considered the overall tail-dependence of $X$ or the tail-dependence function $\lambda(L)$ for 
arbitrary $L\subset [p]$. In this section we will focus on the bivariate dependence as in Section 
\ref{sec:bivariate-intro}. Specifically, we look at the spectral distance and prove that it is both $L^1$- and,
equivalently, $\ell_1$-embeddable. For special spectral distances, namely those corresponding to line metrics, we prove 
that they are rigid and completely determine the tail-dependence of a TM-model. 

\subsection{$L^1$-embeddability of the spectral distance}
Recall that a function $d:T\times T \to [0,\infty)$ on a non-empty set $T$ is called a {\em semi-metric} on $T$
if (i) $d(u,u) = 0$, $u\in T$ (ii) $d(u,v) = d(v,u),\ u,v\in T$ and {(iii)} $d(u,w)\le d(u,v)+d(v,w),\ u,v,w\in T$. The semi-metric is a metric if $d(u,v) = 0$ only if $u=v$.

\begin{definition} A semi-metric $d$ on a set $T$ is said to be $L^1(E,\nu)$-embeddable (or short 
$L^1$-embeddable, when the measure space is understood) if there exists a collection of functions $f_t\in L^1(E,\nu)$,  $t\in T$, such that $$ d(s,t)=\| f_s-  f_t\|_{L^1}=\int_E|f_s(u)-f_t(u)|\nu(du),\ s,t\in T.$$ 
\end{definition}
The concept of $L^1$-embeddability is extensively discussed in \cite{deza:1997}. An overview can also be found in \cite{Matousek13}. 
Our first theorem in this section shows that the spectral distance matrix $d_X$ of a max-stable vector $X$ as defined in \eqref{eq:dist} is $L^1$-embeddable.

\begin{theorem} \label{thm:L1-embed} \phantom{x}\begin{itemize}
    \item [(i)]  For a simple max-stable vector $X$ with bivariate tail-dependence coefficients $\lambda_{i,j} = \lambda_X(i,j)$, the spectral distance
\begin{equation}\label{e:d-via-X}
 d(i,j) := \lambda_{i,i} + \lambda_{j,j} - 2 \lambda_{i,j}
\end{equation} (see Definition~\ref{def:specdist} and \eqref{eq:d-lambda})
is an $L^1$-embeddable semi-metric. 

\item [(ii)] Conversely, for every $L^1$-embeddable semi-metric $d$ 
on $[p]$, there exists a simple max-stable vector $X$ such that \eqref{e:d-via-X} holds with 
$\lambda_{i,j} := \lambda_X(i,j),\ 1\le i,j\le p$.  Moreover, there exists a $c \geq 0$ such that $X$ may be chosen 
to have equal marginal distributions with $\|X_i\|_1 = c, i \in [p]$.

\item [(iii)] The semi-metric $d$ in parts (i) and (ii) is a metric if and only if $\P[X_i \not = X_j]=1$ for all $i\not=j$.
\end{itemize}
\end{theorem}

\begin{proof} 
Part (i): Suppose that $X=(X_i)_{1\leq i\leq p}$ is simple max-stable and let
$f_i\in L_+^1([0,1])$ be as in \eqref{e:de-Haan-fdd}, where for simplicity and without loss of generality we choose $\nu=$Leb. In view of Relation
\eqref{e:theta-lambda-f}, we obtain
$$
\lambda_X(i,j) = \int_{[0,1]} f_i(x) \wedge f_j(x) dx,\ i,j\in [p].
$$
Now the identity $|a-b| = a + b - 2(a\wedge b)$ implies
\begin{align}\label{e:d-via-fi-fj}
d(i,j)&:=  \int_{[0,1]} |f_i(x) - f_j(x)|dx =
       \int_{[0,1]} f_i(x)dx + \int_{[0,1]} f_j(x) dx - 2\int_{[0,1]} f_i(x)\wedge f_j(x)dx \nonumber \\
       &= \lambda_X(i,i) + \lambda_X(j,j) - 2\lambda_X(i,j).
\end{align}
This shows that the semi-metric in \eqref{e:d-via-X} is $L^1$-embeddable.  Note that
$d$ is a metric if and only if $f_i(\cdot)\not =f_j(\cdot)$, almost everywhere, or equivalently 
$X_i\not = X_j$ a.s., for all $i\not=j$.
\\

Part (ii): Suppose now that $d(i,j) = \| g_i - g_j\|_{L^1}$ for some 
$g_i\in L^1(E,\nu),\ i\in [p]$. For simplicity and without loss of generality, we can assume that $(E,{\cal E},\nu) = ([0,1],{\cal B}[0,1],{\rm Leb})$.  Define the function $g^*(x):= \max_{i\in [p]} |g_i(x)|$ and let
$$
f_i(x) = \left\{\begin{array}{ll}
     g^*(2x) - g_i(2x) &,\ x\in [0,1/2]  \\
     g^*(2x-1) + g_i(2x-1) &,\ x\in (1/2,1]. 
\end{array} \right.
$$
This way, we clearly have that the $f_i$'s are non-negative elements of $L^1([0,1])$ 
and
$$
 \|f_i - f_j\|_{L^1} = \|g_i-g_j\|_{L^1}= d(i,j),\ \ i,j\in[p].
$$
Letting $X_i := I(f_i)$ be the extremal integrals defined 
in \eqref{e:de-Haan}, we obtain as in \eqref{e:d-via-fi-fj} that 
$$
d(i,j) = \|f_i-f_j\|_{L^1}= \lambda_X(i,i) + \lambda_X(j,j) - 2\lambda_X(i,j),\ \ i,j\in[p].
$$
This proves the first claim in part (ii).  It remains to argue that (with this
particular choice of $f_i$'s) the scales of the $X_i$'s are all equal. 
Note that $\|X_i\|_1 = \|f_i\|_{L^1}$ and since 
\begin{align*}
&\int_0^{1/2} g^*(2x) - g_i(2x) dx  =  \frac{1}{2}\int_0^1 g^*(u) - g_i(u) du\\
&\int_{1/2}^{1} g^*(2x-1) + g_i(2x-1) dx  =   \frac{1}{2} \int_{0}^1 g^*(u)+g_i(u)du,
\end{align*}
we obtain
$\|X_i\|_1 = \|f_i\|_{L^1}  =  \int_0^{1} g^*(u) du,$ for all $i\in[p],$
which completes the proof of part (ii).\\

Part (iii): The claim follows from the observation that $X_i :=I(f_i)  =  I(f_j)=:X_j$ almost surely 
if and only if $f_i=f_j$ a.e., or equivalently, $\|f_i-f_j\|_{L^1}=0$.
\end{proof}

\begin{remark}
The construction in the proof of part (ii) of Theorem \ref{thm:L1-embed} still works for $f_i$ replaced by $\tilde{f}_i=f_i+\tilde{c}$ for any $\tilde{c}>0$. Thus, the constant $c$ can be chosen equal to or larger than $\int_0^{1} g^*(u) du$, where $g^*(x):= \max_{i\in [p]} |g_i(x)|$ and $g_i\in L^1(E,\nu),\ i\in [p]$ such that $d(i,j) = \| g_i - g_j\|_{L^1}$. In particular, for $\int_0^{1} g^*(u) du\leq 1$, one may choose $X$ with standardized marginals, i.e.\ $\|X_i\|_1=1, i \in [p]$.
\end{remark}

\subsection{$\ell_1$-embeddability of the spectral distance}
\label{sec:line-tm}

In Theorem \ref{thm:L1-embed} we have shown the equivalence between $L^1$-embeddable metrics and spectral distances of 
simple max-stable vectors. In this section, we will additionally state an explicit formula for the $\ell_1$-embedding of 
the spectral distance. Thereby we show that $L^1$- and $\ell_1$-embeddability are equivalent and, in passing, we 
recover and provide novel probabilistic interpretations of the so-called cut-decomposition of $\ell_1$-embeddable metrics \citep{deza:1997}.
\begin{definition}\label{Def:linemetric}
A semi-metric $d$ on $T$ is said to be $\ell_1$-embeddable in $(\R^m, \|\cdot\|_{\ell_1})$ (or short $\ell_1$-embeddable) for some integer $m\geq 1$ if there exist $x_t=(x_t(k))_{1\leq k\leq m}\in \R^m$, $t\in T$, such that $$d(i,j)=\|x_i-x_j\|_{\ell_1}=\sum_{k=1}^{m}|x_i(k)-x_j(k)| \quad \text{ for all } \quad i,j\in T.$$
\end{definition}

\begin{prop}\label{p:d-via-cuts} A semi-metric $d$ on the finite set $[p]$ is embeddable in $L^1(E,{\cal E},\nu)$ if and only if 
\begin{equation}\label{e:d-via-cuts}
d(i,j) = \sum_{J\, :\, \emptyset\not=J\subset [p]} \beta(J) | 1_{J}(i) - 1_{J}(j)|,\ \ i,j\in [p],
\end{equation}
for some non-negative $\beta(J)$'s. 
This means that $d$ is $L^1$-embeddable if and only if it is $\ell_1$-embeddable in $\R^m$,
where $m = |{\cal J}|$ and ${\cal J} = \{ \emptyset \neq J\subset [p]\, :\, \beta(J)>0\}$. 
Indeed, \eqref{e:d-via-cuts} is equivalent to  $d(i,j) = \|x_i - x_j\|_{\ell_1}$, with 
$x_i = (x_i(J))_{J\in {\cal J}} := (\beta(J)1_J(i))_{J\in {\cal J}} \in \R_+^m,\ i,j\in[p]$. 
\end{prop}

\begin{proof}  By Theorem \ref{thm:L1-embed}, $d$ is $L^1$-embeddable if and only if \eqref{e:d-via-X} holds, where $\lambda_{i,j}=\lambda_X(\{i,j\})$ for some simple max-stable
random vector. 
In view of \eqref{e:lambda-via-the-set-Theta} for the special case of $J=\{i\}$, using that 
$\P[J\subset \Theta] = \E[ 1_{\{ J\subset \Theta\}}]$, we have
\begin{equation}\label{e:d(K,L)-via-Theta}
\frac{1}{\theta[p]} \cdot  d(i,j) = \P[i\in \Theta] + \P[j\in \Theta] - 2\P[\{i,j\} \subset \Theta]
=\E[ | 1_{\{i\in \Theta\}} -1_{\{j\in \Theta\}}|].
\end{equation}
Taking $X^*$ to be the (generalized) TM-model with matching extremal coefficients to those of $X$, by Relations \eqref{e:Theta-PMF} and \eqref{e:d(K,L)-via-Theta} we obtain \eqref{e:d-via-cuts}.
\end{proof}

\begin{remark}
Equation \eqref{e:d(K,L)-via-Theta} shows that the spectral distance $d$ is proportional to the probability that
the limiting exceedance set $\Theta$ covers \textit{one and only one} of the points $i$ and $j$.
\end{remark}

\begin{remark}\label{Rem:L1isl1}
Proposition \ref{p:d-via-cuts} recovers the well-known result that $L^1-$ and $\ell_1-$embeddability are equivalent \citep[see Theorem 4.2.6 in][]{deza:1997}.  
\end{remark}

Proposition \ref{p:d-via-cuts} also 
provides a \textit{probabilistic} interpretation of the so-called cut-decomposition of $\ell_1$-embeddable metrics.  To connect to the rich literature on the subject, we will
introduce some terminology following Chapter 4 of the monograph of \cite{deza:1997}. 

Let $J\subset [p]$ be a non-empty set and define the so-called {\em cut semi-metric}:
\begin{equation}\label{e:delta-J}
\delta(J)(i,j) = \left\{ \begin{array}{ll}
 1 &, \mbox{ if $i\not = j$ and } |J\cap \{i,j\}| = 1\\
 0 &,\ \mbox{ otherwise}.
\end{array}\right.
\end{equation}
The positive cone {\rm CUT}$_p:=\{ \sum_{J\subset [p]} c_J \delta(J),\ c_J\ge 0\}$ is referred
to as the {\em cut cone} of non-negative functions defined on $[p]$. Notice that CUT$_p$ consists of semi-metrics. Therefore, Proposition \ref{p:d-via-cuts} entails that the cut cone CUT$_p$ 
comprises all $\ell_1$-embeddable metrics on $p$ points \citep[Proposition 4.2.2 in][]{deza:1997}.  
Relation \eqref{e:d-via-cuts}, moreover, provides  a decomposition of any such metric as 
a positive linear combination of cut semi-metrics.  The coefficients of this decomposition 
are precisely the coefficients of \textit{some} Tawn-Molchanov model.  Finally, in view
of \eqref{e:d(K,L)-via-Theta}, the random exceedance set $\Theta$ of this TM-model is 
such that
\begin{equation}\label{e:d-ij-via-Theta}
d(i,j) = \theta([p]) \cdot \E[ |1_{\Theta}(i) - 1_{\Theta}(j)|].
\end{equation}

\begin{remark}
For a given spectral distance $d$, Proposition \ref{p:d-via-cuts} provides a decomposition and thereby shows the $\ell_1$-embeddability of $d$ in $\R^m$, where $m = |{\cal J}|$ and ${\cal J} = \{ \emptyset \neq J\subset [p]\, :\, \beta(J)>0\}$. Without further knowledge about the number of $J$ such that $\beta(J)>0$ we can always choose $m=2^p-2$, since we may set $\beta([p])=0$ as it does not affect $d$. However, by Caratheodory's theorem each $\ell_1$-embeddable
metric on $[p]$ is in fact known to be $\ell_1$-embeddable in $\mathbb{R}^{\binom{p}{2}}$, see \citep[Proposition 1.4.2]{Matousek13}. We would like to mention that finding the corresponding ``minimal'' TM-model (i.e.\ the one with minimal $|\mathcal{J}|$) and analyzing the properties of such representations could be an interesting topic for further research. 
\end{remark}

Observe that
$$
\delta(J)(i,j) = |1_J(i) - 1_J(j)| = |1_{J^c}(i) - 1_{J^c}(j)| = \delta(J^c)(i,j),\ \ i, j\in [p], 
$$
where $J^c = [p]\setminus J$, which implies that,  in general, the decomposition of $d$ in Proposition \ref{p:d-via-cuts} is not unique. Furthermore, $\beta([p]) \geq 0$ does not affect $d$ in \eqref{e:d-via-cuts}, since $|1_{[p]}(i)-1_{[p]}(j)|=0$. The next definition guarantees that, apart from those unavoidable ambiguities, the representation in \eqref{e:d-via-cuts} is essentially unique.
\begin{definition}\label{d:rigid}
An $\ell_1$-embeddable
metric $d$ is said to be {\em rigid} if for any two representations 
$$d(i,j) = \sum_{J\, :\, \emptyset\not=J\subset [p]} \beta(J) | 1_{J}(i) - 1_{J}(j)|,\ \ i,j\in [p], $$
and
$$d(i,j) = \sum_{J\, :\, \emptyset\not=J\subset [p]} \tilde{\beta}(J) | 1_{J}(i) - 1_{J}(j)|,\ \ i,j\in [p], $$ 
with non-negative $\beta(J), \tilde{\beta}(J), \emptyset \neq J \subset[p],$ the equality
$$ \beta(J)+\beta(J^c)=\tilde{\beta}(J)+\tilde{\beta}(J^c) $$
holds for all $\emptyset \neq J \subsetneq[p]$.
\end{definition}
Observe that each semimetric $d$ on $p$ points can be identified with a vector $d = (d(i,j),\ 1\le i < j\le p)$ in
$\R^N$, where $N:= {p\choose 2}$.  Thus, sets of such semimetrics can be treated as subsets of the Euclidean space
$\R^N$. By Corollary 4.3.3 in \cite{deza:1997}, the metric $d$ is rigid, if and only if it lies on a 
{\em simplex face}  of the cut-cone ${\rm CUT}_p$.  That is, if and only if the set 
$\{J_1,\cdots,J_m\}=\{\emptyset \neq J \subset [p]: \beta(J)>0\}$ is such that the cut semimetrics 
$\delta(J_i),\ i=1,\cdots,m$ (defined in \eqref{e:delta-J}) lie on an
affinely independent face of ${\rm CUT}_p$.  
Recall that the points $\delta_i\in\mathbb R^N,\ i=1,\cdots,m$ are affinely independent  if and only if
$\{\delta_i-\delta_1,\ i=2,\cdots,m\}$ are linearly independent. 
In general, the description of the faces of the cut-cone is challenging, but the next section deals with a special class of metrics which are always rigid. 

\subsection{Rigidity of line metrics}
In this section we show that so-called line metrics are rigid (cf. Definition \ref{d:rigid})
and that for spectral distances corresponding to line metrics the bivariate tail-dependence coefficients, in combination with the marginal distribution, fully determine the higher order tail-dependence coefficients of the underlying random vector and thus the coefficients of the corresponding Tawn-Molchanov model. 
\begin{definition}
A metric $d$ on $[p]$ is said to be a \emph{line metric} if there exist a permutation $\pi=(\pi_i)_{1\leq i\leq p}$ of $[p]$ and some weights $w_k\geq 0$, $1\leq k\leq p-1$, such that $$d(\pi_i,\pi_j)=\sum_{k=i}^{j-1} w_k.$$
\end{definition}
In other words, $d$ is a line metric if all points of $[p]$ can be ordered with different distances on some line and the distance between any two points equals the distance along that line.

\begin{theorem}\label{thm:line-tm} Let $d$ be a line metric, where without loss of generality the indices are ordered in such a way that 
	for all $1\le i < j\le p$ and some $w_k\ge 0$
	\begin{equation}\label{e:d-line}
		d(i,j) = \sum_{k=i}^{j-1} w_k.\ \ 
	\end{equation}
	\begin{itemize}	
	\item[(i)] The line metric $d$ is $\ell_1$-embeddable and rigid.
	\end{itemize}
Assume in addition that $X$ follows a (generalized) TM-model as in \eqref{e:TM-model} with given univariate $\lambda(i) = \lambda_{i,i}$
and bivariate tail-dependence coefficients $\lambda(i,j)=\lambda_{i,j}$ satisfying \eqref{e:d-via-X} with $d$ as in \eqref{e:d-line}. Then:
	\begin{itemize}
	\item[(ii)] For every non-empty set $J\subset [p]$, we have
	\begin{equation}\label{e:lambda-J-line}
		\lambda (J) = \lambda(i,j),\ \ \mbox{ where } i = \min (J) \mbox{ and }j=\max(J).
	\end{equation}
	\item [(iii)] For the coefficients $\beta(J)$ of the (generalized) TM-model, we have that for all 
	$1\le k\le p-1$,
	\begin{equation}\label{e:beta-J-formula}
	 \beta([1:k]) =\lambda(k)-\lambda(k,k+1),\ \ \beta([k+1:p]) 
	 = \lambda(k+1)-\lambda(k,k+1),
	 \end{equation}
	 where $[i:j]:=\{i, i+1, \ldots, j-1,j\}, i<j \in [p],$
	 \begin{equation}\label{e:beta([p])} \beta([p])=\lambda(1,p),\end{equation}
	 and $\beta(J)=0$ for all other $J\subset [p]$.
	\end{itemize}

\end{theorem}

\begin{proof} {\it Part (i):} To see that $d$ is $\ell_1$-embeddable, set $\beta([1:k])=w_k, k \in [p-1],$ and $\beta(J)=0$ for all other sets $\emptyset \neq J \subset [p]$, which gives
  $$ d(i,j)=\sum_{k=i}^{j-1} w_k=\sum_{k=i}^{j-1} \beta([1:k])=\sum_{J\, :\, \emptyset\not=J\subset [p]} \beta(J) | 1_{J}(i) - 1_{J}(j)|=\sum_{J\, :\, \emptyset\not=J\subset [p]} \beta(J) \delta(J).$$
  Thus, $d$ is $\ell_1$-embeddable by Proposition \ref{p:d-via-cuts}. 
  
  Let now $\beta(J), \emptyset \neq J \subset [p]$ be the coefficients of a representation \eqref{e:d-via-cuts} of $d$. We will show that  \begin{equation}\label{e:Jwithbetapos} \beta(J)>0 \;\; \Rightarrow \;\; J=[1:k]  \mbox{ or } J=[k:p] \mbox{ for some } k \in [p]. \end{equation}
  To this end, note that \eqref{e:d-line} implies, for any $i\leq j \in [p]$, that
  $d(i,j)=\sum_{k=i}^{j-1}d(k,k+1)$ and thus
  $$ \sum_{J\, :\, \emptyset\not=J\subset [p]} \beta(J) | 1_{J}(i) - 1_{J}(j)| = \sum_{k=i}^{j-1} \sum_{J\, :\, \emptyset\not=J\subset [p]} \beta(J) | 1_{J}(k) - 1_{J}(k+1)|,$$
  or, equivalently,
  \begin{equation}\label{e:betacoeff0} \sum_{J\, :\, \emptyset\not=J\subset [p]} \beta(J) \left(| 1_{J}(i) - 1_{J}(j)| - \sum_{k=i}^{j-1} | 1_{J}(k) - 1_{J}(k+1)|\right) =0. \end{equation}
  Since 
  $$ \sum_{k=i}^{j-1} | 1_{J}(k) - 1_{J}(k+1)| \geq | 1_{J}(i) - 1_{J}(j)| $$
  and all $\beta(J)$ are non-negative, \eqref{e:betacoeff0} implies that 
  $$ | 1_{J}(i) - 1_{J}(j)| = \sum_{k=i}^{j-1} | 1_{J}(k) - 1_{J}(k+1)| $$ 
  for those $J$ with $\beta(J)>0$ and all $i \leq j \in [p]$. This leads to the following four possible cases:
  \begin{itemize}
      \item [(i)] If $1, p \in J$, then $J=[p]$.
      \item[(ii)] If $1, p \in J^c$, then $J = \emptyset$.
      \item[(iii)] If $1 \in J, p \in J^c$, then there exists one $k \in [p]$ such that $J=[1:k]$.
      \item[(iii)] If $1 \in J^c, p \in J$, then there exists one $k \in [p]$ such that $J=[k:p]$. 
  \end{itemize}
  We have thus shown \eqref{e:Jwithbetapos} and in order to show that $d$ is rigid, we only need to consider sets of the form $J=[1:k], J^c=[k+1:p], k \in [p-1]$. For those sets we get
  \begin{equation}\label{e:betasandws} \beta([1:k])+\beta([k+1:p])= \sum_{J:\emptyset\not=J\subset [p]} \beta(J) | 1_{J}(k) - 1_{J}(k+1)|= d(k,k+1)=w_k, \end{equation}
  and thus the sum $\beta(J)+\beta(J^c)=w_k$ is invariant
  for all representations \eqref{e:d-via-cuts} of $d$ and $d$ is rigid.
  
{\it Part (ii):} Let $\emptyset \neq J \subset [p]$ and set $i=\min(J), j=\max (J)$. Then, from part {\it (i)} and \eqref{e:theta-lambda-via-beta},
\begin{eqnarray*}
\lambda(J)&=&\sum_{K: J \subset K}\beta(K) = \sum_{k \in [p]: J \subset [1:k]}\beta([1:k])+\sum_{k \in [2:p]: J \subset [k:p]}\beta([k:p])\\
&=& \sum_{k=j}^p\beta([1:k])+\sum_{k=1}^i\beta([k:p])\\
&=&\sum_{k \in [p]: i,j \in [1:k]}\beta([1:k])+\sum_{k \in [2:p]: i,j  \in [k:p]}\beta([k:p])=\lambda(\{i,j\})=\lambda(i,j),
\end{eqnarray*}
where we used the fact that $\beta(J) = 0$, for all $J\subset [2:p-1]$ established in
the proof of part {\it (i)}. This completes the proof of {\it (ii)}.\\

{\it Part (iii):}
We have from \eqref{e:betasandws} that
$$ \beta([1:k])+\beta([k+1:p])= d(k,k+1) = \lambda(k)+\lambda(k+1)-2\lambda(k,k+1),$$
and it follows for $k \in [1:p-1]$ by {\it (i)} and \eqref{e:theta-lambda-via-beta} that
\begin{eqnarray*}
\lambda(k)-\lambda(k+1) 
&=& \sum_{k \in J}\beta(J)-\sum_{k+1 \in J}\beta(J) \\
&=&\sum_{j=k}^p\beta([1:j])+\sum_{j=1}^k\beta([j:p])-\sum_{j=k+1}^p\beta([1:j])-\sum_{j=1}^{k+1}\beta([j:p])\\
&=& \beta([1:k])-\beta([k+1:p]).
\end{eqnarray*}
Together, this gives \eqref{e:beta-J-formula}. Furthermore, \eqref{e:beta([p])} follows from
$$ \lambda(1,p)=\sum_{J: 1,p \in J}\beta(J)=\beta([1:p]).$$
That $\beta(J)=0$ if $J$ is not of the form $[1:k]$ or $[k:p], k \in p,$ has already been shown in $(i)$.

\end{proof}
\begin{remark} Consider a max-stable vector $X$ with standard $1$-Fr\'echet marginals, i.e., $\|X_i\|_1 
= \lambda_X(i) = 1,\ i\in [p]$.  Theorem \ref{thm:line-tm} shows that if the spectral distance
$d_X(i,j)= 2(1-\lambda_X(i,j)),\ i,j\in[p]$ is a line metric on $[p]$, then 
$$
\beta([1:k]) =\beta([k+1:p]) = 1-\lambda_X(k,k+1),\ 1\le k\le p-1,\;\; \beta([1:p])=\lambda_X(1,p),
$$
{and for all other $\emptyset \neq J \subset [p], \beta(J)=0$.}
In particular, all higher order extremal  coefficients of $X$ are then completely determined by the bivariate tail-dependence coefficients and
given from \eqref{e:theta-lambda-via-beta} by 
\begin{eqnarray*}
\theta_X(K) &=& \sum_{J:J \cap K \neq \emptyset}\beta(J)=\sum_{j= \min K }^p \beta([1:j])+ \sum_{j=1}^{\max K} \beta([j:p])-\beta([1:p]) \\
&=& \sum_{j= \min K }^p (1-\lambda_X(j,j+1)) + \sum_{j=1}^{\max K} (1-\lambda_X(j,j+1))-\lambda_X(1,p).
\end{eqnarray*}
\end{remark}

\begin{remark}
The random set $\Theta$ corresponding to such line-metric tail-dependence is a random segment with one of its endpoints anchored at $1$ or $p$. This is a direct consequence of the characterisation of $\beta(J)$ in from Theorem \ref{thm:line-tm}~{\it (iii)} and \eqref{e:Theta-PMF}.
\end{remark}

\begin{remark}
In practical applications, the non-parametric inference on higher-order tail-de-pendence coefficients can be very challenging or virtually impossible. Only, say, the bivariate tail-dependence coefficients $\Lambda = (\lambda_X(i,j))_{p\times p}$ of the vector $X$ may be estimated well. Given such constraints, one may be interested in providing upper and lower bounds on $\lambda_X(\{1,\cdots,p\})$, which provide the worst- and best-case scenarios for the probability of simultaneous extremes.  

If the spectral distance turns out to be a line metric and the marginal distributions are known, then Theorem \ref{thm:line-tm} provides a way to precisely calculate $\lambda_X(\{1,\cdots,p\})$. 
However, in general this problem falls in the framework of computational risk management 
\citep[see e.g.][]{embrechts:puccetti:2010} as well as the distributionally robust inference
perspective \citep[see, e.g.][and the references therein]{yuen:2020}.  The problem can be stated as a linear optimization problem in dimension $2^p-1$, similar to the approach in \cite{yuen:2020}.  Unfortunately, the exponential growth of complexity of the problem makes it computationally 
intractable for $p\ge 15$. In fact, the exact solution to such types of optimization problems may be NP-hard. This underscores the importance of the line of research initiated by 
\cite{Shyamalkumar2020} where new approximate solutions or model-regularized approaches to 
distributionally robust inference in high-dimensional extremes are of great interest.
\end{remark}


\section{Computational complexity of decision problems}\label{sec:NP}

In this section we will use known results about the algorithmic complexity of $\ell_1$-embeddings to derive that the so-called tail dependence realization problem is NP-complete, thereby confirming a conjecture from \cite{Shyamalkumar2020}. While a formal introduction to the theory of algorithmic complexity is beyond the scope of this paper, we shall informally recall the basic notions needed in our context following the 
treatment in \citep[Section 2.3]{deza:1997}.

Consider a class of computational problems $D$, where each instance ${\cal I}$ of $D$ can be 
encoded with a finite number of bits $|{\cal I}|$. $D$ is said to be a decision problem, if 
for any input instance ${\cal I}$ there is a correct answer, which is either ``yes'' or ``no''.  
The goal is to determine this answer based on any input ${\cal I}$ by using a computer (i.e., a deterministic Turing machine).

The decision problem $D$ is said to belong to:
\begin{itemize}
    \item The class {\bf P}  (for polynomial complexity), if there is an algorithm (i.e., a deterministic Turing machine), 
that can produce the correct answer in polynomial time, i.e. its running time is of the order $\mathcal{O}(|{\cal I}|^k)$ for some $k \in \mathbb{N}$. 

 \item The class {\bf NP} (nondeterministic polynomial time) if the problem admits a 
 polynomially-verifiable positive certificate.  More precisely, this means that for each 
instance ${\cal I}$ of $D$ with positive (``yes'') answer, there exists a finite-bit certificate ${\cal C}$ of size $|{\cal C}|$ that can be 
verified by an algorithm / deterministic Turing machine with running time $\mathcal{O}(|{\cal C}|^l)$ for some $l \in \mathbb{N}$. (The certificate needs not be constructed in polynomial time.)

\item The class {\bf NP-hard} if \textit{any} problem in NP reduces to $D$ in 
polynomial time. This means that for every problem $D'$ in NP, the correct answer to this decision problem for any instance $\mathcal{I}'$ of $D'$ can be found by first applying an algorithm that runs in polynomial time of $|\mathcal{I'}|$ to transform $\mathcal{I}'$ into an instance $\mathcal{I}$ of $D$ and then solve the decision problem $D$ for this instance $\mathcal{I}$. Note that this definition does not require that $D$ itself is in NP.

\item The class {\bf NP-complete} if $D$ is both in NP and is NP-hard.

\end{itemize}

A decision problem which has received some attention recently, see \cite{fiebig:2017}, \cite{embrechts:hofert:wang:2016}, \cite{krauseetal2018}, and \cite{Shyamalkumar2020}, is the realization problem of a TD matrix with standardized entries on the diagonal, namely finding an algorithm with the following input and output:
\begin{framed}
\underline{Tail dependence realization (TDR) problem} 
\begin{itemize}
	\item Input: A non-negative, symmetric $p \times p$-matrix $L=(L_{i,j})$ with $L_{i,i}=1, i \in [p]$.
	\item Output: An answer to the question: Does there exist a simple max-stable vector $X=(X_1, \ldots, X_p)$ such that
	$$ \lambda_X(i,j)=L_{i,j}, \;\;\; i,j \in [p],$$
	i.e.\ $L$ is the matrix of bivariate tail-dependence coefficients of a max-stable vector with standard $1$-Fr\'{e}chet-margins?
\end{itemize}
\end{framed}
	This problem may at a first glance look similar to deciding whether a given matrix is a valid covariance matrix. Indeed, as a strengthening of Remark~\ref{rem:Bercomp}, it can be shown that there exists a bijection between TD matrices as in the above problem and a subset of the so-called Bernoulli-compatible random matrices, i.e.\ expected outer products $E(YY^t)$ of random (column) vectors $Y$ with Bernoulli margins, see \cite{embrechts:hofert:wang:2016} and \cite{fiebig:2017}. But while it is a simple task to check if a matrix is the covariance matrix of \emph{some} random vector, for example by finding the eigenvalues of this matrix, it can become more difficult to check whether a matrix is the covariance matrix or outer product of a restricted space of random variables. Practical and numerical aspects of deciding whether a given matrix is a TD matrix have been studied in \cite{krauseetal2018} and \cite{Shyamalkumar2020}, including a discussion on the computational complexity of the problem. Indeed, they point out that due to results by \cite{pitowsky1991}, checking whether a matrix is Bernoulli-compatible is an NP-complete problem. However, some subtlety arises as in order to check whether a $p \times p$-matrix $L$ is a so-called tail coefficient matrix, i.e.\ a TD matrix with 1's on the diagonal, it needs to be checked that $p^{-1}L$ is Bernoulli-compatible, see \cite{Shyamalkumar2020}. Thus, the problem narrows down to checking Bernoulli compatibility of the subclass of matrices with $1/p$ on their diagonal and this may have a different complexity than the general membership problem. Due to the similarity in the above mentioned problems, \cite{Shyamalkumar2020} conjecture that the TDR problem is NP-complete as well.
	
	We add to the discussion by using results about computational complexity of problems related to cut metrics and metric embeddings, see Section 4.4 in \cite{deza:1997} for a brief overview over some relevant results. To this end, let us first introduce a problem which is related to the TDR problem but easier to handle for the subsequent complexity analysis.
	
\begin{framed}
 \underline{Spectral distance realization (SDR) problem with unconstrained, identical margins} 
\begin{itemize}
    \item Input: A non-negative, symmetric $p \times p$-matrix $d = (d(i,j))_{p\times p}$.
    \item Output: An answer to the question: Does there exist a simple max-stable vector $X=(X_1, \ldots, X_p)$ and some $c>0$ such that $\|X_i\|_1=c, i \in [p],$ and 
    \begin{equation}\label{Eq:realizationd} d(i,j)= 2 (c-\lambda_X(i,j)),\;\;\; i,j \in [p], \end{equation}
    i.e.\ $d$ is the spectral distance of $X$?
\end{itemize}
\end{framed}
With the help of our previous results and the known computational complexity of $\ell_1$-embeddings it is simple 
to establish the computational complexity of the above problem.
\begin{theorem}\label{Th:SDR:is:NPcomplete}
    The SDR problem with unconstrained, identical margins is NP-complete.    
\end{theorem}
\begin{proof}
 Due to Theorem~\ref{thm:L1-embed} (i)-(ii), the spectral distance $d(i,j)=2( c-\lambda_X(i,j))$ of a simple max-stable random vector with $\|X_i\|_1=c, i \in [p],$ is $L^1$-embeddable and for each $L^1$-embeddable semi-metric $d$ there exists a simple max-stable vector $X$ with $\|X_i\|_1=c, i \in [p],$ for some $c>0$ such that $d$ is the spectral distance of $X$. Thus, the question is equivalent to checking that $d$ is $L^1$-embeddable and this is equivalent to checking that $d$ is $\ell_1$-embeddable, see Remark~\ref{Rem:L1isl1}. The latter problem is NP-complete by \cite{avis:deza:1991}, see also (P5) in \cite{deza:1997}. 
\end{proof}
\begin{remark}
In the SDR problem one could add more assumptions about $d$ in the first place under ``Input'', for example that the entries on the diagonal of $d$ are equal to 0 or that $d$ is a distance matrix. Alternatively, one could also just assume under ``Input'' that $d$ is a $p \times p$-matrix. Since a positive answer to the question would always ensure that $d$ is a distance matrix and all mentioned properties (non-negativity, symmetry, triangle inequality) could be checked in a number of steps which is a polynomial in $p$ these additional assumptions do not change the NP-completeness of the problem.
\end{remark}
Unfortunately, the constant $c$ in \eqref{Eq:realizationd} is not part of the input in the algorithm and thus cannot be fixed a priori. If we could for example set $c=1$ and thus ask if for a given $d$ a simple max-stable vector $X$ with standard $1$-Fr\'{e}chet-margins exists such that $d(i,j)=2(1-\lambda_X(i,j))$, then this is equivalent to checking that $\lambda_{i,j}:=1-d(i,j)/2$ is a TD matrix. But while such an arbitrary fixation of $c$ may change the nature of the problem, the following statement points out an a posteriori feasible range for $c$.

\begin{lemma}\label{lem:SDRrange} If the outcome of the SDR problem with unconstrained, identical margins is a positive answer to the question, then \eqref{Eq:realizationd} holds for a suitable chosen max-stable vector $X$ 
and every $c \ge (2^p-2)\max_{i,j \in [p]}d(i,j)$.
\end{lemma}

The proof is given in Appendix \ref{sec:proofs}.
From the previous lemma we see that the SDR problem with unconstrained, identical margins is equivalent to 
\begin{framed}
 \underline{Spectral distance realization (SDR) problem with $d$-constrained margins} 
\begin{itemize}
    \item Input: A non-negative, symmetric $p \times p$-matrix $d = (d(i,j))_{p\times p}.$
    \item Output: An answer to the question: Does there exist a simple max-stable vector $X=(X_1, \ldots, X_p)$ such that $\|X_i\|_1=(2^p-2)\max_{i,j \in [p]}d(i,j)
    , i \in [p],$ and 
$$ d(i,j)=2((2^p-2)\max_{i,j \in [p]}d(i,j)-\lambda_X(i,j)),\;\;\; i,j \in [p], $$
    i.e.\ $d$ is the spectral distance of $X$?
\end{itemize}
\end{framed}
Finally, by changing from $X$ to $\tilde{X}:=X/((2^p-2)\max_{i,j \in [p]}d(i,j))$ the spectral distance $d_{\tilde{X}}$ of $\tilde{X}$ and bivariate tail-dependence coefficients $\lambda_{\tilde{X}}(i,j)$ scale accordingly by Lemma~\ref{lemma:d-linear} and we see that the latter problem is actually equivalent to 
\begin{framed}
 \underline{Spectral distance realization (SDR) problem with constrained, standard margins} 
\begin{itemize}
    \item Input: A non-negative, symmetric $p \times p$-matrix $d= (d(i,j))_{p\times p}.$
    \item Output: An answer to the question: Does there exist a simple max-stable vector $X=(X_1, \ldots, X_p)$ such that $\|X_i\|_1=1, i \in [p],$ and 
\begin{equation}\label{e:lambdaij-via-dij}
 \frac{d(i,j)}{(2^p-2)\max_{i,j \in [p]}d(i,j)}=2(1-\lambda_X(i,j)),\;\;\; i,j \in [p],
 \end{equation}
    i.e.\ $\lambda(i,j):=1-d(i,j)/(2(2^p-2)\max_{i,j \in [p]}d(i,j))$ is the matrix of bivariate tail-dependence coefficients of a max-stable vector with standard $1$-Fr\'{e}chet-margins?
\end{itemize}
\end{framed}

From the last line in the above problem we can see that our SDR problem with constrained, standard margins can be solved if we have an algorithm to check that $\lambda$ of the given form is a TD matrix. But since we know by the stated equivalence of all three SDR problems in combination with Theorem~\ref{Th:SDR:is:NPcomplete} that all of them are NP-complete, we know that this algorithm has to be NP-complete as well. This leads to the following result.
\begin{theorem}
The TDR problem is NP-complete.
\end{theorem}

\begin{proof}

We need to show that the TDR problem is both in NP and NP-hard. That the TDR problem is in NP has been shown in \citep[p.\ 255]{Shyamalkumar2020}, with the help of Caratheodory's theorem. We start with the first statement and follow the typical way to prove this by reducing a known NP-complete problem to TDR. Indeed, any input matrix $d(i,j)$ to any of the three equivalent, and by Theorem~\ref{Th:SDR:is:NPcomplete} NP-complete, SDR problems can be transformed in polynomial time to the matrix $\lambda(i,j):=1-d(i,j)/(2(2^p-2)\max_{i,j \in [p]}d(i,j))$. By the statement of the third SDR problem, the question with input $d$ can be answered by using $\lambda$ as an input to the TDR problem. Thus, an NP-complete problem reduces in polynomial time to the TDR problem and the TDR problem is NP-hard, thus NP-complete.

\end{proof}

\bibliographystyle{agsm}
\bibliography{refs.bib}

\appendix
\section{ Proofs and auxiliary results} \label{sec:proofs}

\subsection{Proofs for Section \ref{sec:notation} }
\begin{proof}[Proof of Proposition \ref{p:h(Y)}] Here, for brevity, we shall write $\{h\in A\}$
for the pre-image set $h^{-1}(A)=\{x\in \R^p\, :\, h(x)\in A\}$. By the continuity of $h$, it follows that for all $a\ge 0$, the set
$\{h\ge a\}$ is closed and $\{h>a\}$ is open. Hence $\{h=a\} = \{h\ge a\}\setminus \{h>a\} \supset \partial \{h>a\}$, 
where $\partial A = A^{\rm cl}\setminus A^{\rm int}$
denotes the boundary of the set $A$. Since $h(0) = 0$ (by continuity and homogeneity), we have that
for all $a>0$, the closed set $\{h\ge a\}$ does not contain $0$ and hence it is bounded away from $0$.  Thus, $\mu(\{ h\ge a\}) <\infty$.   Since
$\{h=t\} = t \cdot\{h=1\},\ t>0$, the scaling property \eqref{e:mu-scaling} 
of $\mu$ implies that $\mu(\{h=t\}) = t^{-\alpha} \mu(\{h=1\})$ and if 
$\mu(\{h=t\})>0$ for some (any) $t>0$, then $\mu(\{h=t\})>0$, for all $t>0$.  On the other hand, we have 
that $\{h\ge a\} = \cup_{t\ge a} \{h=t\}$, where the latter union involves an uncountable
collection of disjoint sets. Thus, $\mu(\{h=t\})$ must vanish
for all $t>0$.  This means that $\mu(\partial\{\mu> a\}) = 0$, or that $\{h>a\}$ are $\mu$-continuity sets for all $a>0$.
This allows us to apply the definition of regular variation \eqref{e:d:RV} and obtain
$$
n\P[h(X)> a_n] = n\P[X\in a_n\cdot \{h>1\} ] \to \mu(\{h>1\}),\ \mbox{ as }n\to\infty.
$$
Now, \eqref{e:mu-spectral} entails
\begin{align*}
\mu(\{h>1\}) &= \int_S \int_0^\infty 1_{\{ h(ru)>1\}} \alpha r^{-\alpha-1} dr \sigma(du)\\
& = \int_S \int_0^\infty 1_{\{ h(u)> 1/r\}} \alpha r^{-\alpha-1} dr \sigma(du)\\
& =  \int_S \int_0^\infty 1_{\{ h^\alpha(u) > x \}} dx  \sigma(du)=\int_S h^\alpha(u) \sigma(du),
\end{align*}
where in the last two displays we used the homogeneity of $h$ and the change of variables $x:=r^{-\alpha}$.

This completes the proof of the first relation in \eqref{e:p:h(Y)}.  The second relation therein follows from the
observation that $\sigma(\cdot)/\sigma(S)$ is a probability distribution.
\end{proof}

\begin{proof}[Proof of Proposition \ref{p:X-max-stable-is-RV}]
Since $X$ has non-negative components, to establish its regular variation, it is enough consider
measures supported only on $\R_+^p$ and show that 
$$
\mu_n (\cdot):= n \P[ X \in n\cdot]\ToM0 \mu,\ \mbox{ as }n\to\infty,
$$
where $\mu[0,x]^c = -\log(\P[X\le x])$ with $[0,x]^c:= \R_+^p \setminus [0,x]$.

Fix an $x \in [0,\infty)^p\setminus\{0\}$. The definition of simple max-stability \eqref{e:max-stability-def} entails
$\P[X \le  n x]^n = \P[X\le x]$.  Thus, for all $n \in \mathbb{N}$,
\begin{equation}\label{e:p:X-max-stable-is-RV-1}
\P[ X\le nx ]^n = \Big(1 - \frac{n \P[ X\in n\cdot [0,x]^c]}{n} \Big)^n = \Big( 1-\frac{\mu_n[0,x]^c}{n}\Big)^n 
= \P[X\le x].
\end{equation}
Observe that $\P[X\le x]$ is \textit{positive}.  Indeed, it is easy to see that since $X$ is non-degenerate
$\xi:=\max_{1\leq i\leq p} x_i^{-1} X_i$ is $1$-Fr\'echet and thus $\P[\xi\leq 1 ] = \P[X\le x] >0$, for all $x\in \R_+^p$.
This means that for the boundedly finite measures  $\mu_n(\cdot)= n \P[ X\in n\cdot]$,  we have 
\begin{equation}\label{e:p:X-max-stable-is-RV-2}
 \lim_{n\to\infty} \mu_n([0,x]^c) = -\log \P[X\le x] .
\end{equation}
The latter relation shows that the sequence of measures $\{\mu_n\}$ is relatively compact in $M_0(\R^p)$, equipped with 
the $M_0$-convergence topology.  Indeed, by \citep[Theorem 2.7]{hult:lindskog:2006}, it suffices to show that for all $\varepsilon>0$ and $\eta>0$, there exists an $M=M(\varepsilon,\eta)>0$, such that
$$
\sup_{n} \mu_n( B(0,\varepsilon)^c) <\infty\ \ \mbox{ and }\ \ \sup_{n} \mu_n\big(\R_+^p\setminus [0,M]^p \big) <\eta.
$$
The first condition follows from \eqref{e:p:X-max-stable-is-RV-2} and since $\P[X \leq x]>0$. 
The second condition follows from the fact that
$-\log(\P[X\le x]) \downarrow 0$, as $x\uparrow \infty$, which is true since $X$ has a valid probability distribution.

The relative compactness of the measures $\{\mu_n\}$ entails that $\mu_{n'}\ToM0 \mu$ for some $\mu\in M_0$ and 
a sub-sequence $n'\to \infty$.  However,  by \eqref{e:p:X-max-stable-is-RV-2} and Proposition \ref{p:de-Haan} we have 
\begin{align}\mu[0,x]^c &= -\log \P[X\le x] \\
&= \int_{E} \max_{1\leq i\leq p} \frac{f_i(u)}{x_i} \nu(du) \\
\label{mu:nu:f}&= \int_{E} \int_0^\infty 1_{[0,x]^c}(r \vec{f}(u))r^{-2}dr \nu(du)
\end{align}
for all 
$x\in [0,\infty)^p\setminus\{0\},$ and the limit measure is uniquely determined by its values on all the complements of 
rectangles containing the origin. Furthermore, we see from \eqref{mu:nu:f} that for non-degenerate $X$, the limit measure $\mu$ is non-degenerate as well. This proves that $X\in RV(\{n\},\mu)$ where $\P[X\le x] = \exp\{-\mu[0,x]^c\}$. 

Having established regular variation, the first equality in Relation \eqref{e:p:max-stable-RV} follows from the
$\mu$-continuity of the set $\{h>1\}$ as argued in the proof of Proposition \ref{p:h(Y)}.  The rest of Relation \eqref{e:p:max-stable-RV} follows from \eqref{mu:nu:f}.

Finally, the representation in \eqref{e:p:max-stable-RV-sigma} follows from the fact that $\sigma$ is determined by 
$\int_S g(u) \sigma(du)$, for all continuous functions $g: S\to \R_+$.  Indeed, for every such $g$, 
the function $h(x):= g(x/\|x\|) \|x\| 1_{\{x\not = 0\}}$ is continuous, non-negative and $1$-homogeneous and hence by 
\eqref{e:p:max-stable-RV}
$$
\int_S g(u) \sigma(du) = \int_S g\Big( \frac{\vec f(z)}{\|\vec f(z)\|}  \Big) \|\vec f(z)\| \nu(dz).
$$
This, since $g$ is arbitrary, proves \eqref{e:p:max-stable-RV-sigma}.
\end{proof}

\begin{proof}[Proof of Corollary \ref{c:theta-lambda-mu}]
    Relation \eqref{e:theta-lambda-mu} follows by applying \eqref{e:p:max-stable-RV} to $h$ replaced by the continuous and homogeneous functions $h_{\max,L} (x):= (\max_{i\in L}x_i)_+$ and  $h_{\min,L} (x):= (\min_{i\in L}x_i)_+$, respectively.  Indeed, observe that
    $$
    n \P\Big[\min_{i\in L} X_i>n\Big] = n\P\Big[h_{\min,L}(X)>n\Big]\to \mu(\{h_{\min,L}>1\}) = \lambda(L),
    \ \ \mbox{ as }n\to\infty.
    $$
    On the other hand, $\{x\, :\, h_{\min,L}(x)>1\} = \bigcap_{i\in L} A_i$. 
 
    The formula for $\lambda(L)$ in 
    Relation \eqref{e:theta-lambda-f} follows from the above and Equation \eqref{e:p:max-stable-RV} since
     $
     h_{\min,L}(\vec{f})=\min_{i\in L} f_i(x).
     $
     The derivations of the formulae for $\theta(K)$ are similar.
\end{proof}

We conclude this section with the auxiliary result, that the spectral distance and the tail-dependence coefficients are linear under max-linear combinations, in the sense of the following lemma.
\begin{lemma}\label{lemma:d-linear}
Let $X^{(t)}=(X^{(t)}_i)_{1\leq i\leq p}$, $1\leq t\leq n$, be independent simple max-stable vectors with tail measures $\mu^{(t)}$, $1\leq t\leq n $, and let $\gamma_t\geq 0$, $1\leq t\leq n$, be some non-negative weights. Define $\bar{X}=\bigvee_{t=1}^{n}\gamma_t X^{(t)}$.
Then, $$d_{\bar{X}}(i,j)=\sum_{t=1}^{n}\gamma_t d_{X^{(t)}}(i,j) \quad \text{ for all } \quad i,j\in[p]$$ and
$$\lambda_{\bar{X}}(L)=\sum_{t=1}^{n}\gamma_t\lambda_{X^{(t)}}(L) \quad \text{ for all } \quad L\subset [p].$$ 
\end{lemma}

\begin{proof}[Proof of Lemma \ref{lemma:d-linear}]
    By the independence of $X^{(t)}$, $1\leq t\leq n$, and Proposition \ref{p:X-max-stable-is-RV} it applies
\begin{align}
    &\P\big[\bar{X}\leq x \big]
    = \prod_{t=1}^{n}\P\Big[X^{(t)}\leq \frac{1}{\gamma_t}x\Big]
    = \prod_{t=1}^{n}\exp\Big\{ - \mu^{(t)}\big[0,\frac{1}{\gamma_t}x\big]^c  \Big\}
    =\exp\Big\{ - \sum_{t=1}^{n} \gamma_t \mu^{(t)}[0,x]^c \Big\}
\end{align} 
for all $x \in \mathbb{R}_+^p \setminus\{0\}$, where in the last step the homogeneity of $\mu^{(t)}$ was applied.
Thus, $\bar{X}$ has the tail measure $\mu_{\bar{X}}=\sum_{t=1}^{n}\gamma_t \mu^{(t)}$, i.e. the tail measure of the max-linear combination $\bar{X}$ is the corresponding linear combination of the tail measures of the components. 
In particular, $\bar{X}$ has $1$-Fr\'echet marginals with scale coefficient $\|\bigvee_{t=1}^{n}\gamma_t X^{(t)}_i\|_1=\sum_{t=1}^{n}\gamma_t\|X^{(t)}_i\|_1$. 

Hence, by the definition of the spectral distance $d_{\bar{X}}$ in Definition \ref{def:specdist} we obtain 
\begin{align}\label{eq:d-decomp-gen}
    d_{\bar{X}}(i,j)&=2 \Big\|\bigvee_{t=1}^{n}\gamma_tX^{(t)}_i\vee \bigvee_{t=1}^{n}\gamma_tX^{(t)}_j\Big \|_1 - \Big\|\bigvee_{t=1}^{n}\gamma_tX^{(t)}_i\Big\|_1-\Big\|\bigvee_{t=1}^{n}\gamma_tX^{(t)}_j\Big\|_1\\
    &= 2\sum_{t=1}^{n}\gamma_t\|X^{(t)}_i\vee X^{(t)}_j\|_1 - \sum_{t=1}^{n}\gamma_t\|X^{(t)}_i\|_1 -\sum_{t=1}^{n}\gamma_t\| X^{(t)}_j\|_1 =  \sum_{t=1}^{n}\gamma_t d_{X^{(t)}}(i,j).
\end{align}
By the linear representation of $\mu_{\bar{X}}$ and \eqref{e:theta-lambda-mu} it follows for the tail-dependence coefficients
\begin{align}\label{eq:lambda-decomp-gen}
    \lambda_{\bar{X}}(L)=\mu_{\bar{X}} \Big(\bigcap_{i\in L} A_i \Big) = \sum_{t=1}^{n}\gamma_t\mu^{(t)} \Big(\bigcap_{i\in L} A_i \Big) = \sum_{t=1}^{n}\gamma_t \lambda_{X^{(t)}}(L)
\end{align}
\end{proof}

\subsection{Proofs for Section \ref{sec:TM}}

\begin{lemma}\label{l:formulae} For $\lambda$ and $\theta$ in 
\eqref{e:theta-lambda-via-beta} and $\beta$ in \eqref{e:beta-via-mu}, we have the inversion 
formulae \eqref{e:beta-via-theta}, \eqref{e:beta-via-lambda} as well as \eqref{e:theta-and-lambda}.  Namely, the following formulae hold:
\begin{align*}
&\beta(J) = \sum_{K\, :\, \emptyset \not=K,\ J^c\subset K} (-1)^{|J\cap K|+1} \theta(K),  
 &\beta(J) = \sum_{L\, :\, J \subset L \subset [p]} (-1)^{|L\setminus J|} \lambda(L) \\
&\theta(K) = \sum_{L\, :\, \emptyset\not=L\subset K}(-1)^{|L|-1}\lambda(L),
 &\lambda(L) = \sum_{K\, :\, \emptyset \not=K\subset L} (-1)^{|K|-1} \theta(K). 
\end{align*}
\end{lemma}
\begin{proof} For simplicity, introduce the indicator functions 
$I_i:= 1_{A_i}$, where $A_i = \{x\in \R^p\, :\, x_i>1\}$.  In view of 
\eqref{e:beta-via-mu} and \eqref{e:theta-lambda-mu}, we have
$$
\beta(J) = \int\Big( \prod_{i\in J} I_i \times \prod_{j\in J^c} (1-I_j) \Big) d\mu
$$
as well as
\begin{equation}\label{e:lambda-via-I}
\lambda(L) = \int \Big( \prod_{i\in L} I_i \Big) d\mu.
\end{equation}
This immediately entails
\begin{align*}
\beta(J) &= \int \Big( \sum_{J\subset L \subset [p]} (-1)^{|L\setminus J|} 
\prod_{i\in L} I_i  \Big) d\mu =\sum_{L\, :\, J\subset L \subset [p]} (-1)^{|L\setminus J|} \lambda(L),
\end{align*}
which proves \eqref{e:beta-via-lambda}.

The inclusion-exclusion formula for $\theta$ in terms of the $\lambda$'s is immediate from \eqref{e:lambda-via-I} and the observation that 
$$
\theta(K) = \int \Big(1 - \prod_{i\in K} (1-I_i) \Big) d\mu 
= \sum_{L\, :\,\emptyset \neq  L\subset K} (-1)^{|L|-1} \int \Big(\prod_{i\in L} I_i \Big) d\mu = \sum_{L\, :\, \emptyset \neq L\subset K} (-1)^{|L|-1} \lambda(L).
$$
Now, using \eqref{e:lambda-via-I} we obtain
\begin{align}\label{e:lambda-via-indicators-towards-theta}
\lambda(L) &= \int \Big(\prod_{i\in L} (1-(1-I_i)) \Big) d\mu\nonumber \\
& = \int \Big(1+\sum_{K\, :\, \emptyset \not = K\subset L} (-1)^{|K|} \prod_{i\in K} (1-I_i) \Big) d\mu.
\end{align}
Observe that by Newton's binomial formula:
$$
0= (1+(-1))^{|L|} = \sum_{K\subset L} (-1)^{|K|} = 1 + \sum_{K\, :\, \emptyset \not=K\subset L} (-1)^{|K|},
$$
and hence
$$
1 = \sum_{K\, :\, \emptyset \not=K\subset L} (-1)^{|K|-1}.
$$
Using the latter expression for the constant $1$ in the right-hand side of \eqref{e:lambda-via-indicators-towards-theta}, we obtain
$$
\lambda(L) = \sum_{K\, :\, \emptyset \not=K\subset L} (-1)^{|K|-1} \int \Big( 1- \prod_{i\in K}(1-I_i) 
\Big)d\mu = \sum_{K\, :\, \emptyset \not=K\subset L} (-1)^{|K|-1} \theta(K),
$$
completing the proof of the inclusion-exclusion formula for the $\lambda$'s via the $\theta$'s in
\eqref{e:theta-and-lambda}.

To complete the proof we need to establish the expression of $\beta(J)$'s via the $\theta(K)$'s in \eqref{e:beta-via-theta}.
We do so by passing through the $\lambda(L)$'s first.  Namely, by the established 
\eqref{e:beta-via-lambda} and \eqref{e:theta-and-lambda}, we have
\begin{align*}
    \beta(J) &= \sum_{L\, :\, J\subset L \subset [p]} (-1)^{|L\setminus J|} \lambda(L)\\
    & = \sum_{L\, :\, J\subset L \subset [p]} (-1)^{|L\setminus J|} \Big( \sum_{K\, :\, \emptyset \not=K\subset L} (-1)^{|K|-1} \theta(K) \Big)\\
    & = \sum_{K\, :\, \emptyset \not=K \subset [p]} \theta(K) (-1)^{|J|+|K|-1 - |J\cup K| }
    \times \Big(
    \sum_{L\, :\, (K\cup J) \subset L}  (-1)^{|L| - |J\cup K| } \Big) \\
    & =: \sum_{K\, :\, \emptyset \not=K \subset [p]} \theta(K) (-1)^{|J\cap K|+1} \times C(K,J).
\end{align*}
Observe that 
$C(K,J):= \sum_{L\, :\, (J\cup K)\subset L} (-1)^{|L|-|J\cup K|} =0$ if $(J\cup K)^c \not = \emptyset.$  Indeed, the latter sum is simply $(1+(-1))^{| [p] \setminus (K\cup L) |} = 0$.  On the other hand, if 
$J\cup K = [p]$, we trivially have $C(K,L)=1$.  This, since $J\cup K = [p]$ is equivalent to $J^c\subset K$, immediately implies 
$$
\beta(J) = \sum_{K\, :\, \emptyset \not=K,\ J^c\subset K} \theta(K) (-1)^{|J\cap K|+1}.
$$
This proves \eqref{e:beta-via-theta}.
\end{proof}

\begin{proof}[Proof of Theorem \ref{t:TM-model}] For this proof it is convenient to let
$\|u\| := \max_{i\in [p]} |u_i|$ be the sup-norm.

(`if') Suppose that all $\beta(J)$'s in \eqref{e:beta-via-theta} (or in \eqref{e:beta-via-lambda})
are non-negative and define $X^*$ as in \eqref{e:TM-model}. Clearly, $X^*$ is max-stable and we shall determine its
spectral measure $\sigma^*$ in the sup-norm.  Observe that for all $x\in (0,\infty)^p$ we have
$\beta(J) Z_J 1_J\le x$, if and only if $Z_J \le \min_{i\in J} x_i /\beta(J)$ and since the $Z_J$'s are iid 
standard $1$-Fr\'echet:
$$
\P[ X^*\le x] = \exp\Big\{ - \sum_{\emptyset \neq J\subset [p]} \frac{\beta(J)}{\min_{i\in J} x_i} \Big\}.
$$
Letting
$
\sigma^*(du) := \sum_{\emptyset \neq J\subset [p]} \beta(J) \delta_{1_J}(du),
$
we obtain that 
\begin{equation}\label{e:beta-sigma}
\int_{S} \Big(\max_{i\in [p]} \frac{u_i}{x_i} \Big) \sigma^*(du) = \sum_{\emptyset \neq J\subset [p]} \frac{\beta(J)}{\min_{i\in J} x_i}.
\end{equation}
This shows that 
$$
\P[X^*\le x] = \exp\{-\mu^*[0,x]^c\} =  \exp\Big\{ - \int_S \Big(\max_{i\in [p]} \frac{u_i}{x_i} \Big) \sigma^*(du)\Big\},
$$ 
which in view of \eqref{e:max-stable-CDF-via-sigma} shows that $\sigma^*$ is the spectral measure of $X^*$, where $\mu^*$
denotes the tail measure of $X^*$.

Let now $\emptyset\not=J\subset[p]$ and define the set
$$
C_J:=\{ u\in S\, :\, u_i = 1\, \mbox{ for all $i\in J$ and $u_j=0$, for all $j\in J^c$ }\}.
$$
We will argue that, with $B_J, A_j$ as in \eqref{e:beta-via-mu},
\begin{equation}\label{e:CJ_BJ}
\sigma^*(C_J) = \mu^*(B_J):= \mu^* \Big( \bigcap_{i\in J} A_i \cap \bigcap_{j\in J^c} A_j^c \Big).
\end{equation}
Indeed, by \eqref{e:sigma-spectral}, we have 
$$
\sigma^*(C_J) = \mu^*(\widetilde C_J):= 
\mu^*\{ x\in \R_+^d\, :\, \|x\|>1,\ x/\|x\|\in C_J\}.
$$
Note that if $x\in \widetilde C_J$, then $x_i=\|x\|>1$ for all $i\in J$, and
$x_j = 0 <1$, for all $j\in J^c$, so that $x\in B_J$.  This means that 
$\widetilde C_J\subset B_J$, and hence $\sigma^*(C_J) \le \mu^*(B_J)$.  

By the construction of $\sigma^*$, on the other hand, we have $\sigma^*(C_J) = \sigma^*(\{1_J\}) = \beta(J)$ and since the $B_J$'s partition the set 
$\{\|x\|>1\}\cap \R_+^p$, we get
$$
\mu^*(\{\|x\|>1\})=\mu^*(\{\|x\|>1\}\cap \R_+^p) = \sum_{J} \mu^*(B_J) \ge \sum_J \sigma^*(C_J) = \sum_{J} \beta(J) = \sigma^*(S),
$$
where the last relation follows from \eqref{e:beta-sigma} by setting $x_i=1, i \in [p]$.
Since $\sigma^*(S) = \mu^*(\{\|x\|>1\})$ we obtain from the above inequality that $\mu^*(B_J) = \sigma^*(C_J)=\beta(J)$, for all $J\subset[p]$.

We have thus shown that the functionals $\beta(J)$ that we started with are indeed the ones which determine the extremal (tail dependence) coefficients of $X^*$ via \eqref{e:theta-lambda-via-beta}. This completes the proof of the
`if'  part.\\

(`only if') Conversely, if $\{\theta(K),\ K\subset[p]\}$ (or $\{\lambda(L),\ L\subset[p]\}$)
are the extremal coefficients (tail-dependence coefficients, respectively) of a max-stable
vector $X$ with tail measure $\mu$.  Then, as already argued above \eqref{e:theta-lambda-mu} holds, and hence the $\beta(J)$'s defined as in \eqref{e:beta-via-mu} are non-negative and 
satisfy Relations
 \eqref{e:beta-via-theta} (or \eqref{e:beta-via-lambda}). This completes the proof.
\end{proof}

\subsection{Proofs for Section \ref{sec:NP}}
\begin{proof}[Proof of Lemma \ref{lem:SDRrange}]
Assume that for an input matrix $d$ for the SDR problem with unconstrained, identical margins the answer is positive, which is equivalent to $d$ being $L^1$-embeddable, see the proof of Theorem~\ref{Th:SDR:is:NPcomplete}. According to Proposition~\ref{p:d-via-cuts} and Lemma~\ref{lemma:d-linear} we can then choose the realizing max-stable vector $X$ as a (generalized) TM-model with coefficients $\beta(J), \emptyset \neq J \subset [p]$, such that
$$ c=\sum_{J \ni i} \beta(J), \;\;\; i \in [p],$$
and
\begin{equation}\label{Eq:d:and:beta}d(i,j) = \sum_{J\, :\, \emptyset\not=J\subset [p]} \beta(J) | 1_{J}(i) - 1_{J}(j)|,\ \ i,j\in [p].\end{equation}
We note that the particular choice of $\beta([p]) \geq 0$ does only affect the value of $c$ and is not determined by \eqref{Eq:d:and:beta} as $|1_{[p]}(i)-1_{[p]}(j)|=0$. Thus, for each realizing max-stable vector $X$ with $\|X_i\|=c$ we can for each $\tilde{c}>c$ find a realizing max-stable vector $X$ with $\|X_i\|=\tilde{c}$ by increasing the value of $\beta([p])$ in the generalized TM-model.

Since all $\beta(J)'s$ are non-negative, \eqref{Eq:d:and:beta} implies furthermore that for all $\emptyset \neq J \subsetneq [p]$ the inequality
$$ \beta(J) \leq \max_{i,j \in [p]}d(i,j)$$
holds. Thus we know that for all $i \in [p]$
$$ c = \sum_{J \ni i} \beta(J)\leq \beta([p])+(2^p-2)\max_{i,j \in [p]}d(i,j)$$
As any choice of $\beta([p])\geq 0$ leads to a realizing (generalized) TM-model for given $d$ we see that any value $c \geq (2^p-2)\max_{i,j \in [p]}d(i,j)$
is possible for the marginal scale of this model. 
\end{proof}

\end{document}